\documentclass[12pt,reqno,a4paper]{amsart}
\usepackage{extsizes}
\usepackage{blindtext}
\usepackage{fullpage}
\usepackage{mathtools}
\usepackage{amsmath,amssymb,amsthm}
\usepackage{amscd}
\usepackage{bm}
\usepackage{hyperref}
\usepackage{dsfont}
\usepackage{enumerate}
\usepackage{epsfig}
\usepackage{float, graphicx}
\usepackage{latexsym, amsxtra}
\usepackage{mathrsfs}
\usepackage{multicol}
\usepackage[normalem]{ulem}
\usepackage{psfrag}
\usepackage[parfill]{parskip}
\usepackage{stmaryrd}

\usepackage[T1]{fontenc}
\usepackage{url}
\usepackage{verbatim}
\usepackage{indentfirst}
\usepackage{tikz-cd}
\usepackage[all]{xy}

\usepackage{mathtools}
\usepackage{cases}

\makeatletter
\def\thm@space@setup{%
\thm@preskip=2ex \thm@postskip=2ex
}
\makeatother

\hypersetup{hidelinks}

\newtheorem{thm}{Theorem~}[section] 
\newtheorem{lem}[thm]{Lemma~}

\newtheorem{prop}[thm]{Proposition~}

\newtheorem{cor}[thm]{Corollary~}

\theoremstyle{remark}
\newtheorem{rmk}[thm]{Remark~}
\newtheorem{examp}[thm]{Example~}

\theoremstyle{definition}
\newtheorem{defn}[thm]{Definition~}

\newcommand{\CC}{\mathbb{C}}
\newcommand{\ZZ}{\mathbb{Z}}
\newcommand{\RR}{\mathbb{R}}

\newcommand{\PP}{\mathbb{P}}

\newcommand{\QQ}{\mathbb{Q}}
\newcommand{\DD}{\mathbb{D}}

\newcommand{\BB}{\mathbb{B}}

\newcommand{\AAA}{\mathbb{A}}

\newcommand{\Prd}{\mathscr{P}}
\newcommand{\SF}{\mathscr{F}}

\newcommand{\La}{\mathscr{L}}

\newcommand{\calC}{\mathcal{C}}

\newcommand{\calE}{\mathcal{E}}

\newcommand{\calV}{\mathcal{V}}
\newcommand{\calO}{\mathcal{O}}

\newcommand{\calF}{\mathcal{F}}

\newcommand{\calT}{\mathcal{T}}
\newcommand{\calU}{\mathcal{U}}

\newcommand\Aut{\mathrm{Aut}}
\newcommand\Gal{\mathrm{Gal}}

\newcommand\Proj{\mathrm{Proj}}
\newcommand\Sym{\mathrm{Sym}}
\newcommand\SL{\mathrm{SL}}

\newcommand\id{\mathrm{id}}

\newcommand{\rank}{\mathrm{rank}}
\newcommand{\Pic}{\mathrm{Pic}}
\newcommand{\diag}{\mathrm{diag}}

\newcommand{\GL}{\mathrm{GL}}

\newcommand{\Bl}{\mathrm{Bl}}
\newcommand{\Ree}{\mathrm{Re}}
\newcommand{\MW}{\mathrm{MW}}
\newcommand{\disc}{\mathrm{disc}}

\newcommand{\FS}{\mathfrak{S}}

\newcommand{\bs}{\backslash}
\newcommand{\dbs}{\bs\!\! \bs}

\title{The complex ball-quotient structure of the moduli space of certain sextic curves}
\vspace{1cm}
\author[Z. Zheng]{Zhiwei Zheng}
\address{Yanqi Lake Beijing Institute of Mathematical Sciences and Applications, Beijing, China}
\email{zhengzw@bimsa.cn}

\author[Y. Zhong]{Yiming Zhong}
\address{Tsinghua University, Beijing, China}
\email{zhongym16@mails.tsinghua.edu.cn}
\date{}

\begin{document}
	
\begin{abstract}
We study moduli spaces of certain sextic curves with a singularity of multiplicity $3$ from both perspectives of Deligne-Mostow theory and periods of K3 surfaces. In both ways we can describe the moduli spaces via arithmetic quotients of complex hyperbolic balls. We show in Theorem \ref{theorem: main} that the two ball-quotient constructions can be unified in a geometric way. 
\end{abstract}
	
\maketitle

\section{Introduction}
When study moduli spaces of certain complex curves or surfaces, we are naturally led to the moduli spaces of weighted points on $\PP^1$. Here are several such examples:

\begin{examp}
\label{example: 1}
Kond\=o \cite{kondo2000moduli} considered the moduli space of non-hyperelliptic curves of genus $4$. A generic non-hyperelliptic curve of genus $4$ is the zero locus of a section of $\calO(3)\boxtimes \calO(3)$ on $\PP^1\times \PP^1$. The projection of the curve to $\PP^1$ (there are two such projections) is a triple cover with $12$ branch points. 
We also know that an elliptic fibration on a rational surface has $12$ branch points. See \cite{heckman2002moduli} for work on such rational surfaces.
Via these constructions, the moduli spaces of non-hyperelliptic curves of genus $4$, of $12$ points with the same weight, and of rational elliptic surfaces are closely related to each other.
\end{examp}

\begin{examp}
\label{example: 2}
Allcock, Carlson and Toledo \cite{allcock2002complex} realized the moduli space of cubic surfaces as an arithmetic ball quotient of dimension $4$ via the intermediate Jacobian of certain associated cubic threefolds. After that, Dolgachev, van Geemen and Kond\={o} \cite{dolgachev2005complex} studied the moduli space of cubic surfaces from a different point of view. They started with a line $l$ lying on a smooth cubic surface $S$, and looked at the planes containing $l$. There are five such planes intersecting with $S$ in $3$ lines, and two such planes intersecting with $S$ in $l$ and a quadric curve tangent to $l$. In this way we obtain $5+2$ points on the pencil of planes containing $l$. They then associated a plane sextic curve with an automorphism of order $3$ to each set of $2+5$ points. Via these constructions, the moduli spaces of cubic surfaces, of  $2+5$ points on $\PP^1$ with certain weights, and of the corresponding sextic curves are closely related to each other.
\end{examp}

\begin{examp}
\label{example: 3}
Kond\={o} \cite{kondo2007moduli5} considered del Pezzo surfaces of degree $4$.
The quadrics in $\PP^4$ containing a fixed del Pezzo surface of degree $4$ form a pencil with $5$ branch points.
Each set of $5$ points gives a plane sextic curve with an automorphism of order $5$.
Via these constructions, the moduli spaces of del Pezzo surfaces of degree $4$, of $5$ points on $\PP^1$ with same weight, and of the corresponding sextic curves are closely related to each other.
\end{examp}

Moreover, in all these examples, there are K3 surfaces associated with sets of points on $\PP^1$ in a natural way.
Each of these K3 surfaces contains a pencil of curves of genus one or two, together with an action of cyclic group preserving all members of the pencil. 
Moreover, the set of points on $\PP^1$ is exactly the set of singular members in the pencil.
The moduli spaces of the sets of branch points are all ball quotients and can be constructed by periods of the corresponding $K3$ surfaces.

There is another natural way to deal with moduli spaces of points on $\PP^1$, namely, the Deligne-Mostow theory (\cite{deligne1986monodromy}, \cite{mostow1986generalized,mostow1988discontinuous}).
We observe that some cases in Deligne-Mostow's list naturally lead to singular sextic curves, as we have seen in Example \ref{example: 2} and \ref{example: 3}.
One of the most interesting new case is the case $(\frac{1}{3}, \frac{1}{3}, \frac{1}{3},\frac{1}{6},\frac{1}{6}, \frac{1}{6}, \frac{1}{6}, \frac{1}{6},\frac{1}{6})$ in Deligne-Mostow’s list, which gives rise to a ball quotient of dimension $6$.
In this paper we study the moduli of $9$ points on $\PP^1$ with this weight.

Denote by $m: [9]\to \PP^1$ an injective map from $[9]\coloneqq \{1,2,\cdots,9\}$ to $\PP^1$. 
On one hand, from Deligne-Mostow's work, one can associate with $m$ a cyclic cover of $\PP^1$ branched along $m([9])$. We call the cover a Deligne-Mostow curve.
The periods of this curve give rise to a ball-quotient structure on the moduli space (see \S \ref{subsection: Review of Deligne-Mostow's Theory}). 
On the other hand, we can attach a K3 surface with a $D_4$ singularity to each $m$.
The periods of the resolutions of those K3 surfaces give rise to another arithmetic  ball quotient (see \S \ref{subsection: period map Prd_W}). 
Our main goal of this paper is to formulate and prove such a geometric unification of those two constructions (see Theorem \ref{theorem: main result} and Diagram (\ref{main diagram})). 

\begin{rmk}
Dolgachev and Kond\={o} \cite{dolgachev2007moduli} summarized some works on the complex ball uniformizations of the moduli spaces of del Pezzo surfaces, $K3$ surfaces and algebraic curves of low genus, and conjectured that all the ball quotients in Deligne-Mostow theory are moduli spaces of certain K3 surfaces.
Moonen \cite{moonen2018deligne-mostow} proved that a large number of ball quotients in Deligne-Mostow's works are moduli of K3 surfaces.
\end{rmk}

The corresponding sextic curves in our case are defined by polynomials of the form $F\coloneqq X_0^3 F_3(X_1,X_2)+F_6(X_1,X_2)$, where $F_i$ are degree $i$ homogeneous polynomials in $X_1,X_2$.
The double cover of $\PP^2$ branched along $Z(F)$ is a singular K3 surface with a $D_4$-singularity and a natural action of $\mu_6$ (the group of $6$-th roots of unity).
Its minimal resolution is a smooth K3 surface $W_F$ with an elliptic fibration over $\PP^1$. 
The quotient of $W_F$ by the involution is a rational surface, which can be obtained from $\PP^2$ by consecutive blowups. See Proposition \ref{proposition: alternative way to get the smooth K3} for more details. A generic such K3 surface $W_F$ has Picard lattice isomorphic to $U\oplus A_2(-1)^3$ and transcendental lattice isomorphic to $A_2\oplus E_6(-1)^2$. 
We will describe the $\mu_6$-actions on $\Pic(W_F)$ and $T(W_F)$ in an explicit way, see \S \ref{section: abstract description}. The idea of the proof is inspired by Kond\=o \cite{kondo2000moduli} and Dolgachev-van Geemen-Kond\={o} \cite{dolgachev2005complex}.
We give a geometric description for the fibration structure on $W_F$. 
There are $9$ singular fibers, $3$ of which are of Kodaira type \uppercase\expandafter{\romannumeral4} and the other $6$ are of Kodaira type \uppercase\expandafter{\romannumeral2}.  
In \S \ref{subsection: trivialization} we show that this fibration is an isotrivial family, and the pullback of this family to Deligne-Mostow curve is birationally a product of two curves.
This geometric construction leads to an identification between two complex hyperbolic balls, one is constructed from the Deligne-Mostow theory and the other one from periods of $K3$ surfaces (see the end of \S \ref{section: hodge structure}). Such an identification is needed in the formulation of our main Theorem \ref{theorem: main}.

Consider the periods of those $K3$ surfaces, we define a period map 
\begin{equation*}
\Prd_S\colon \calF_S\to \Gamma_S\backslash \BB_S,
\end{equation*} 
where $\calF_S$ is the moduli space of the singular sextic curves $Z(F)$, and $\Gamma_S\backslash \BB_S$ is an arithmetic ball quotient of dimension $6$, see \S \ref{subsection: period map Prd_W}. We will prove $\Prd_S$ is an open embedding. Our proof is standard, relying on the global Torelli theorem and some lattice-theoretic analysis. 
On the other hand, from Deligne-Mostow theory, we have another period map which is also an open embedding:
\begin{equation*}
\Prd_{DM}\colon \calF_{DM}\to \Gamma_{DM}\backslash \BB_{DM}.
\end{equation*}
Here, $\calF_{DM}$ is the moduli space of $m\colon [9]\to \PP^1$, and $\Gamma_{DM}\backslash \BB_{DM}$ is an arithmetic ball quotient of dimension $6$. Our main Theorem \ref{theorem: main} relates the above two period maps in a natural way.

\noindent\textbf{Structure of the paper}: In \S \ref{section: Sextics and K3 Surfaces}, we describe the elliptic fibration on the $K3$ surface $W_F$ associated with $F$. 
In \S \ref{section: gitdm}, we establish the relation between the Deligne-Mostow moduli space $\calF_{DM}$ of $m\colon [9]\to \PP^1$ and the moduli space $\calF_S$ of the singular sextic curves $Z(F)$. 
In \S \ref{section: trivialization}, we show (with explicit calculations) that the $K3$ surfaces $W_F$ are birational to quotients of products of two curves. 
In \S \ref{section: period map by K3} we define the period maps for the moduli space $\calF_S$ via periods of $W_F$, and show its injectivity.
In \S \ref{section: hodge structure}, we study the relation between the weight-two Hodge structures on $W_F$ and the weight-one Hodge structures on the Deligne-Mostow curves using results of \S \ref{section: trivialization} and the Chevalley-Weil formula. In particular, we obtain the idenfication between the two balls $\BB_{DM}$ and $\BB_S$.
In \S \ref{section: commu} we summarize our results in a commutative diagram \eqref{main diagram} and prove our main Theorem \ref{theorem: main}.
Finally in \S \ref{section: abstract description} we give an explicit description of the $\mu_3$-action on the transcendental lattice for a generic $W_F$.

\textbf{Acknowledgement}: We thank Eduard Looijenga for stimulating discussion and many helpful comments, and Dali Shen, Chenglong Yu for related discussion, especially on Deligne-Mostow theory. The first author thanks Max Planck Institute for Mathematics for its support and excellent research atmosphere during his stay.

\noindent\textbf{Notation and Conventions:} \\
\noindent 1. $\zeta_n=\exp(\frac{2\pi\sqrt{-1}}{n})$\\
\noindent 2. $\mu_n=\langle \zeta_n\rangle$: the group of $n$-th roots of unity\\
\noindent 3. $\PP V=(V-0)/\CC^{\times}$: the projectivization of a complex vector space $V$\\
\noindent 4. $\CC[X_0, \cdots, X_n]$ and $\CC[X_0, \cdots, X_n]_d$: the ring of polynomials and the space of polynomials of degree $d$\\
\noindent 5. $\Sym^d V$: $d$-th symmetric product of a vector space $V$\\
\noindent 6. $Z(F)$: the zero locus in $\PP V$ of a homogeneous polynomial $F$\\
\noindent 7. $\BB(T)$: the complex hyperbolic ball associated with a unitary Hermitian form $(T, h)$ \\
\noindent 8. $A_L$: the discriminant group of a lattice $L$\\
\noindent 9. For a module $V$ over a ring $R$ and an extension $R\hookrightarrow R'$, we write $V_{R'}=V\otimes_R R'$.\\
\noindent 10. $L_{K3}$: the K3 lattice.

\section{Singular Sextics and $K3$ Surfaces}
\label{section: Sextics and K3 Surfaces}
In this section we first associate $K3$ surfaces with certain singular plane sextic curves, and then study some natural isotrivial fibration structures on the $K3$ surfaces. Polynomials and varieties are defined over the complex field $\CC$.

Let $V$ be a complex vector space of dimension $3$, and let $\PP V$ be the associated projective space of dimension $2$. Let $F\in \Sym^6(V^{*})$ be a sextic polynomial on $V$.  
We denote by $Z(F)$ the sextic curve defined by $F$ in $\PP V$. 
Let $S_F$ be the double cover of $\PP V$ branched along $Z(F)$. 
It is well-known that, when $Z(F)$ is a smooth sextic curve, the surface $S_F$ is a smooth $K3$ surface. 
Suppose $Z(F)$ is a singular sextic curve with only ADE singularities, then $S_F$ is a $K3$ surface with ADE singularities. 
In such situation we denote by $W'_F$ the minimal model of $S_F$, which is a smooth $K3$ surface.

We consider an action of $\mu_3$ on $V$ such that the generator $\zeta_3\in \mu_3$ has eigenvalues  $\zeta_3$ (with multiplicity $1$) and $1$ (with multiplicity $2$). Let $V=V_1\oplus V_2$ be the corresponding decomposition into eigenspaces with $\dim V_1=1$ and $\dim V_2=2$. We denote by $p\coloneqq \PP V_1$ the point in $\PP V$ defined by $V_1$. The projective line $P:=\PP V_2$ can be regarded as the pencil of lines passing through $p$. The $\mu_3$-action on $V$ induces one on $\PP V$.
We consider sextic polynomials $F$ which are invariant under this action. 
Then $p$ is a point of multiplicity $3$ on $Z(F)$. 
We take coordinates $X_0, X_1, X_2$ for $V$ such that $V_1=\{X_1=X_2=0\}$ and $V_2=\{X_0=0\}$. Now we can write $F=F(X_0, X_1, X_2)=X_0^3 F_3(X_1, X_2)+F_6(X_1, X_2)$, where we usually denote by $F_i=F_i(X_1,X_2)$ a homogeneous polynomial in variables $X_1,X_2$ of degree $i$. 
We denote by $\calV$ the space of such sextic polynomials $F$, and let $\calV^{\circ}$ be the subspace of $\calV$ consisting of $F$ such that $F_3 F_6$ has no multiple roots. If not particularly mentioned, we always take $F\in \calV^{\circ}$. 

\begin{lem}
\label{lemma: only singularity}
The point $p$ is the only singularity of $Z(F)$ for $F\in \calV^{\circ}$. 
\end{lem}
\begin{proof}
	Let $[X_0:X_1:X_2]$ be a singularity of $Z(F)$, namely
    \begin{subequations}  \label{eq:1}
    	\begin{align}  
    	3 X_0^2 F_3(X_1,X_2)  &= 0,  \label{eq:1A} \\
    	X_0^3 \frac{\partial F_3}{\partial X_1} + \frac{\partial F_6}{\partial X_1} &= 0,  \label{eq:1B} \\
    	X_0^3 \frac{\partial F_3}{\partial X_2} + \frac{\partial F_6}{\partial X_2} &= 0.  \label{eq:1C}
    	\end{align}
    \end{subequations}
    From $X_1\cdot$\eqref{eq:1B}$+X_2\cdot$\eqref{eq:1C} we obtain 
    \begin{equation}\label{eq:2}
    	3X_0^3 F_3(X_1,X_2)+6 F_6(X_1,X_2)=0.
    \end{equation}
    By \eqref{eq:1A} and \eqref{eq:2}, we obtain $F_6(X_1,X_2)=0$.
    Hence we have $X_1=X_2=0$ or $F_3(X_1,X_2)\ne 0$.
    If $X_1=X_2=0$, then the singularity is the point $p$.
    If $F_3(X_1,X_2)\ne0$, then by $\eqref{eq:1A}$ we obtain $X_0=0$.
    By \eqref{eq:1B} and \eqref{eq:1C}, we obtain $\frac{\partial F_6}{\partial X_1}=\frac{\partial F_6}{\partial X_2}=0$.
    Since $F_6$ does not have multiple roots, we obtain a contradiction.
    Hence $p$ is the only singularity of $Z(F)$ for $F\in \calV^{\circ}$.
\end{proof}

We give an alternative construction of $W'_F$ as follows.
The sextic curve $Z(F)$ is singular at $p$ with multiplicity three. Thus the preimage of $p$ in $S_F$ is a $D_4$-singularity (this is implied from \cite[Theorem 2.23]{greuel2007introduction}).
We first blow up the point $p$ on $\PP V$ and denote the exceptional divisor by $E_p$. 
The strict transform of $Z(F)$, denoted by $K$, intersects with $E_p$ at 3 points, say $q_1,q_2, q_3$. 
We then blow up $q_1,q_2, q_3$ on $\Bl_{p} (\PP V)$ and denote by $E_1, E_2, E_3$ the corresponding exceptional divisors.
Let $\widehat{K}, \widehat{E}_p\subset R\coloneqq \Bl_{\{q_1, q_2, q_3\}} (\Bl_p(\PP V))$ be the strict transforms of $K$ and $E_p$. 
Let $W_F$ be the double cover of $R$ branched along $\widehat{K}$ and $\widehat{E}_p$. 
Since $\widehat{K}\cap\widehat{E}_p=\emptyset$, the surface $W_F$ is smooth. 
\begin{prop}
	\label{proposition: alternative way to get the smooth K3}
	The evident birational map between $W'_F$ and $W_F$ extends to an isomorphism.
\end{prop}
\begin{proof}
We denote by $Q=\Bl_p(\PP V)$, and denote by $g\colon R\rightarrow Q$ and $f\colon Q\to \PP V$ the two blowups. We first compute the canonical class of $R$. Let $L$ be a generic line in $\PP V$. Then the canonical class of $Q$ is given by 
\begin{equation*}
K_Q=f^*[-3L]+[E_p].
\end{equation*}
Let $K_R$ be the canonical class of $R$. Then
\begin{equation*}
K_R=g^*K_Q+[E_1]+[E_2]+[E_3]=g^*f^*[-3L]+g^*[E_p]+[E_1]+[E_2]+[E_3]
\end{equation*}
After rearrangement, we have
\begin{equation}
\label{equ: KR}
K_R=-3g^*f^*[L]+[\widehat{E}_p]+2([E_1]+[E_2]+[E_3])
\end{equation}	
We next compute the class of $\widehat{K}+\widehat{E_p}$ in $R$. Note that $Z(F)$ passes through $p$ with multiplicity $3$. The class $g^*f^*[Z(F)]$ is equal to
	\begin{align*}
	g^*([K]+3[E_p])
	&=[\widehat{K}]+[E_1]+[E_2]+[E_3]+3[\widehat{E}_p]+3([E_1]+[E_2]+[E_3])\\
	&=[\widehat{K}]+3[\widehat{E}_p]+4([E_1]+[E_2]+[E_3]),
	\end{align*}
This implies that
\begin{equation}
\label{equation: pullback of Z(F) and E_p}
[\widehat{K}]+[\widehat{E}_p]=6g^*f^*[L]-2[\widehat{E}_p]-4([E_1]+[E_2]+[E_3]).
\end{equation}
	
Since $W_F$ is the double cover of $R$ branched along $\widehat{K}\sqcup \widehat{E}_p$, the canonical class $K_{W_F}$ of $W_F$ is equal to $2K_R+[\widehat{K}]+[\widehat{E}_p]$. From Equations \eqref{equ: KR} and \eqref{equation: pullback of Z(F) and E_p}, we conclude that $K_{W_F}=0$.
By Lemma \ref{lemma: only singularity}, $W_F$ has no other singularities, thus $W_F$ is a smooth $K3$ surface. 

By the construction, we know that $W_F, W'_F$ are smooth resolutions of $S_F$, hence naturally isomorphic.
\end{proof}

We have a rational map $\pi_F\colon W_F\dashrightarrow P$ which is the composition of the double cover $W_F\longrightarrow \PP V$ and the rational morphism 
$\PP V\dashrightarrow P$, $[X_0:X_1:X_2]\longmapsto [X_1:X_2]$. 
The preimage of $p$ in $S_F$ is the only singularity, which is blown up in the resolution $W_F\to S_F$. 
Hence the rational map $\pi_F$ is automatically a morphism. We describe $\pi_F$ explicitly in the following proposition.
\begin{prop}
	\label{proposition: fibration}
	The morphism $\pi_F\colon W_F\longrightarrow P$ is an elliptic fibration such that every smooth fiber is an elliptic curve with j-invariant $0$. In particular, $\pi_F$ is an isotrivial fibration.
\end{prop}
\begin{proof}
The equation of the singular $K3$ surface $S_F$ in the weighted projective space is 
\begin{equation}
\label{equation: SF}
S^2=X^3_0 F_3(X_1,X_2)+F_6(X_1,X_2)
\end{equation}
where $(S, X_0, X_1, X_2)$ is the weighted homogeneous coordinate system for $\PP(3,1,1,1)$. 
We take a point $a=[a_1: a_2] \in P$. Let $L_a$ be the line in $\PP V$ defined by $a_2 X_1- a_1 X_2=0$. The fiber $\pi_F^{-1}(a)$ is a double cover of $L_a$ branched at $L_a\cap Z(F)$. 
For a generic choice of $a$, the branch locus $L_a\cap Z(F)$ contains four different points (with one point being $p$). 
Therefore, $\pi_F^{-1}(a)$ is a smooth elliptic curve. Therefore, $\pi_F$ is an elliptic fibration with a section $\widehat{E}_p$.

Recall that we have a $\mu_3$-action on $\PP V$ and $Z(F)$, hence on $W_F$. This induced action fixes the section $\widehat{E}_p$ and preserves every fiber of $\pi_F$.
Thus the elliptic curve $\pi_F^{-1}(a)$ has $j$-invariant $0$. 
This implies that $\pi_F$ is an isotrivial fibration.
\end{proof} 

Next we give an explicit calculation of the structure of $\pi_F^{-1}(a)$ for $a=[a_1: a_2]$. We assume $a_1\ne 0$. Then we work in the open subspace $\{X_0\ne 0, X_1\ne 0\} \subset \PP(3,1,1,1)$. Let $s\coloneqq \frac{S}{X_0^2 X_1}$, $x_1\coloneqq \frac{X_1}{X_0}$ and $t\coloneqq \frac{X_2}{X_1}$, then the affine equation of $S_F$ (defined by Equation \eqref{equation: SF}) can be written as
\begin{equation*}
s^2 = x_1 f_3(t) + x_1^4 f_6(t).
\end{equation*} 
Here $f_i(t)=F_i(1, t)$.
Thus the preimage of $L_a$ in $S_F$ is given by the affine equation 
\begin{equation*}
s^2 = x_1 f_3(\frac{a_2}{a_1}) + x_1^4 f_6(\frac{a_2}{a_1}).
\end{equation*} 
The smooth projective model for this affine equation is the unique elliptic curve with $j$-invariant $0$.

Next we describe the elliptic fibration structure on $W_F$ in more details.
	For $a=[a_1:a_2]\in P$, let $\widehat{L}_{a}$ be the strict transform of the line $L_{a}=\{[X_0: X_1: X_2]\big{|} a_2 X_1=a_1 X_2\}\subset \PP V$ in $R=\Bl_{\{q_1, q_2, q_3\}} (\Bl_p(\PP V))$. 
	
	Then $\widehat{L}_{a}$ intersects $\widehat{K}$ in three points and  $\widehat{E_p}$ in one point.
If these four points are distinct (which happens if and only if $a\in P-Z(F_3 F_6)$, see Lemma \ref{lemma: intersect} below), the preimage of $\widehat{L}_{a}$ in $W_F$ is a double cover of $\widehat{L}_{a}$ branched over four distinct points, hence a smooth genus $1$ curve that is preserved by the action of $\mu_3$ on $W_F$. 
	The composition $W_F\rightarrow R\rightarrow Q\rightarrow P$ is hence an elliptic  fibration with a fiberwise action of $\mu_3$. 

	Let $a^1,a^2, a^3$ be the zeros of $F_3$ and $b^1,b^2,\dots, b^6$ the zeros of  $F_6$. Recall that $a^i=[a_1^i: a_2^i]$ and $b^j=[b_1^j: b_2^j]$. By straightforward calculation we obtain:
\begin{lem}
\label{lemma: intersect}
The line $L_{a^i}$ ($1\le i\le 3$) intersects $Z(F)$ at $p=[1:0:0]$ with multiplicity $6$. 
The line $L_{b^j}$ ($1\le j\le 6$) intersects $Z(F)$ at $p=[1:0:0]$ and another point $[0:b_1^j:b_2^j]$, both  with multiplicity $3$. 
Other lines $L_{a}$ for $a\in P-Z(F_3 F_6)$ intersect $Z(F)$ at $p$ with multiplicity $3$, and at three other points with multiplicity one.
\end{lem}
 
    Each $\widehat{L}_{a^i}$ intersects with $\widehat{K}$ with multiplicity $2$, and does not meet $\widehat{E}_p, E_1, E_2$ and $E_3$.
    Each $\widehat{L}_{b^j}$ intersects $\widehat{K}$ with multiplicity $3$ and $\widehat{E}_p$ with multiplicity $1$, and does not meet $E_1, E_2$ and $E_3$.
    The preimage of $\widehat{L}_{a^i}$ in $W_F$ is the union of two projective lines and the preimage of $E_i$ is a projective line. The union of three lines which intersect at one common point is a singular fiber of Kodaira type IV.
    The preimage of $\widehat{L}_{b^j}$ is a cuspidal curve, which is a singular fiber of Kodaira type II.
We sum up the above discussion as follows.
\begin{prop}
\label{proposition: sigular fiber type}
The discriminant set of the elliptic fibration $W_F\to P$ is $Z(F_3 F_6)\subset P$. The fibers over $Z(F_3)$ are of Kodaira type IV (the union of $3$ smooth rational curves intersecting at one point), and the fibers over $Z(F_6)$ are of Kodaira type II (cuspidal cubic curves). 
\end{prop}

\begin{rmk}\label{remark: weierstrass fibration} 
	In fact, given an elliptic K3 surface $W$ with a section and smooth fibers isomorphic to the elliptic curve with $j$-invariant $0$. 
	Assume $W$ has $9$ singular fibers, among which $3$ are of Kodaira type \textup{IV} and $6$ are of Kodaira type \textup{II}. 
	By the discussion in \cite[Chapter $11$, $2.2$]{huybrechts2016lectures}, the Weierstrass model of $W\to \PP^1$ can be written as
	$$
	y^2 z = x^3 + a(t)x z^2 + b(t)z^3,
	$$
	where $a(t)\in H^0(\PP^1,\calO_{\PP^1}(8))$, $b(t)\in H^0(\PP^1,\calO_{\PP^1}(12))$ and $t$ is the coordinate of the base $\PP^1$.
	The $j$-invariant of the fiber over $t\in \PP^1$ is $\frac{1728\cdot 4a(t)^3}{4a(t)^3+27b(t)^2}$, which equals to zero in our case. Thus $a(t)=0$ for all $t\in \PP^1$.
	Hence the Weierstrass model of $W\to \PP^1$ has the following form
	$$
	y^2 z = x^3 + b(t)z^3.
	$$
	
	The order of vanishing of $4a(t)^3+27b(t)^2$ at $t$ is determined by the type of the fiber over $t$, see \cite[Chapter $11$, $2.4$]{huybrechts2016lectures}. Explicitly, $b(t)$ has $3$ zeros of order $2$ (corresponding to fibers of type IV) and $6$ ordinary zeros (corresponding to fibers of type II).
	Therefore, $W$ is uniquely determined up to isomorphism by the position of $9$ discriminant points. 
Let $F=X^3_0 F_3+F_6$, where the zeros of $F_3$ ($F_6$ resp.) coincide with the $3$ zeros of order $2$ ($6$ ordinary zeros resp.) of $b(t)$. Then $W_F$ has also such a fibration structure with discriminant set coincide with that of $W$. Thus $W$ and $W_F$ are isomorphic. 
\end{rmk}
	

\section{GIT Constructions and Deligne-Mostow Theory}
\label{section: gitdm}
In this section we study the GIT-model for sextic curves $Z(F)=V(X_0^3 F_3+F_6)$. 
This is similar to certain constructions in \cite{yu2018moduli} of moduli spaces of nodal sextic curves. 
In our case, this approach leads to a GIT-model for weighted points on $\PP^1$. We shall explain a natural relation between our construction and the Deligne-Mostow theory, see Proposition \ref{proposition: two git}.

\subsection{A GIT Construction}
\label{subseciton: GIT moduli of type I}
Recall that $\calV$ is the vector space of sextic polynomials $X_0^3 F_3+F_6$ and  $\calV^\circ$ is the subset of $\calV$ consisting of elements such that $F_3 F_6$ has no multiple roots. 
Let $\calV_i$ be the vector space of polynomials $X_0^{6-i} F_i$. 
We have $\calV=\calV_3 \oplus \calV_6$. 
We denote by $\PP \calV$ the projectivization of $\calV$ and by $\PP \calV^\circ\subset \PP\calV$ the open subset defined by $\calV^\circ$.

Let $\CC[X_0, X_1, X_2]_6$ be the space of sextic polynomials in $X_0, X_1, X_2$. Define $g(F)=F\circ g^{-1}$ for any $g\in \SL(3, \CC)$ and $F\in \CC[X_0, X_1, X_2]_6$. 
We then have an action of $\SL(3, \CC)$ on $\CC[X_0, X_1, X_2]_6$. 
Let $G_S$ be the $\SL(3, \CC)$-stabilizer of the decomposition $\CC^3=\CC\oplus \CC^2$, which is generated by $\diag(1, g)$ for $g\in \SL(2, \CC)$ and elements $\diag(t^2, t^{-1}, t^{-1})$ for $t\in \CC^{\times}$. Then $G_S$ leaves $\calV\subset \CC[X_0, X_1, X_2]_6$ invariant. 
Note that $G_S$ is naturally isomorphic to $\GL(2)$, hence reductive.
The group $Z=\{\diag(t^2, t^{-1}, t^{-1})\big{|} t\in \CC^{\times}\}$ is the center of $G_S$. Consider the action of $G_S$ on the polarized variety $(\PP \calV, \calO(1))$. 
The points in $\PP \calV^\circ$ are stable under this action (by Shah \cite{shah1980complete}). 

Let $\calF_S\coloneqq G\dbs \PP \calV^\circ$ be the GIT quotient. We will see that $\calF_S$ is naturally isomorphic to an open subspace of an arithmetic ball quotient. 
We postpone the discussion of this aspect to \S \ref{section: hodge structure}, and focus on the GIT-model now. 
The following proposition will lead to a relation between $\calF_S$ and the moduli space of $9$ weighted points on $\PP^1$ from the perspective of Deligne-Mostow theory (see \S \ref{subsection: Review of Deligne-Mostow's Theory}).

\begin{prop}
\label{proposition: two git}
We have an isomorphism 
\begin{equation}
\label{equation: quotient by center}
Z\dbs (\PP \calV, \calO(1))\cong (\PP \calV_3\times \PP \calV_6, \calO(2)\boxtimes \calO(1))
\end{equation}
of polarized varieties. 
\end{prop}

\begin{proof}
	We need to calculate the invariant subalgebra of $\bigoplus_{w=1}^{\infty} \Sym^w \calV^*$ under the induced action of $Z$. We have $\calV= \calV_3\oplus \calV_6$ and $\calV^*= \calV_3^*\oplus \calV_6^*$. 
	Hence
	\begin{equation*}
	\Sym^w \calV^*=\bigoplus_{i+j=w} \Sym^i \calV_3^*\otimes \Sym^j \calV_6^*,
	\end{equation*}
	where $i, j$ are non-negative integers. 
	For $g_t=\diag(t^2, t^{-1}, t^{-1})\in Z$, we have $g_t(X_0^{6-k} F_k)=t^{3k-12} X_0^{6-k} F_k$. 
	Thus the action of $g_t$ on $\calV_3^*$, $\calV_6^*$ is by scalars $t^{-3}, t^6$ respectively.
	Hence the action of $g_t$ on $\Sym^i \calV_3^*\otimes \Sym^j \calV_6^*$ is by the scalar $t^{6j-3i}$. Thus a nonzero element in $\Sym^i \calV_3^*\otimes \Sym^j \calV_6^*$ is invariant under the action of $Z$ if and only if $3i-6j=0$. 
	We have
	\begin{equation*}
	(\bigoplus_{w=1}^{\infty} \Sym^w \calV^*)^Z=\bigoplus_{j=1}^{\infty} (\Sym^{2j} \calV_3^*\otimes \Sym^{j} \calV_6^*)
	\end{equation*}
	By taking $\Proj$-constructions on both sides, we obtain the isomorphism \eqref{equation: quotient by center}
\end{proof}

\begin{rmk}
	A similar result was previously obtained by Yu and Zheng in \cite[Proposition 6.5]{yu2018cubic} to study moduli spaces of cubic fourfolds with specified group actions.
\end{rmk}

\subsection{Deligne-Mostow Theory} 
\label{subsection: Review of Deligne-Mostow's Theory}
In this section we briefly recall the Deligne-Mostow theory, see \cite{deligne1986monodromy}, \cite{mostow1986generalized,mostow1988discontinuous}, \cite{looijenga2007uniformization}. 
Let $N\ge 5$ be a positive integer, and $(\alpha_1, \cdots, \alpha_N)$ be a tuple of rational positive numbers such that $\alpha_1+\cdots+\alpha_N=2$ and $0<\alpha_i<1$. 
Let $d$ be the lowest common multiple of the denominators of $\alpha_1,\cdots, \alpha_N$. 
Recall $P=\PP V_2$ is a projective line. Let $[N]\coloneqq \{1,2,\cdots,N \}$. 
Let $\FS_\alpha$ be the group of permutations of $[N]$ which leave the map  $\alpha: [N]\to \QQ$ invariant. 

We denote by $(P^N)^\circ$ the space of all injective maps from $[N]$ to $P$.
For $m\in (P^N)^{\circ}$ we have distinct linear forms $l_i\in V_2^{*}$ such that $[Z(l_i)]=m(i) \in P$.
The equation
\begin{equation*}
	y^d = \prod_{i=1}^{N}l_i^{d\alpha_i}
\end{equation*} 
defines a curve in the weighted projective space $\PP(2,1,1)$, with normalization $C_m$ called the Deligne-Mostow curve.
Let $b_{C_m}$ be the symplectic bilinear form on $H^1(C_m, \CC)$ given by the cup product. We denote by $\mu_d\subset \CC^{\times}$ the group of $d$-th roots of unity. The curve $C_m$ admits a natural action by $\mu_d$, with the generator $\zeta_d\coloneqq \exp(2\pi \sqrt{-1}/d)$ sending $y$ to $\zeta_d y$. 
Thus $\mu_d$ has an induced action on the cohomology group $H^1(C_m, \ZZ)$. 
This action diagonalizes if take $\QQ(\zeta_d)$ as coefficient: we can decompose $H^1(C_m, \QQ(\zeta_d))$ into characteristic subspaces with respect to this action. 
For a character $\rho\colon \mu_d\to \CC^{\times}$, we denote by $H^1(C_m, \QQ(\zeta_d))_{\rho}$  its characteristic subspace in $H^1(C_m, \QQ(\zeta_d))$. 
There is a natural Hermitian form $h_{\rho}$ on $H^1(C_m, \QQ(\zeta_d))_{\rho}$ defined by $h_{\rho}(x, y)=\frac{\sqrt{-3}}{3} b_{C_m}(x, \overline{y})$. 
Let $\rho_1$ be the natural inclusion of $\mu_d$ into $\CC^{\times}$. We write $T_{C_m}=H^1(C_m, \QQ(\zeta_d))_{\rho_1}$ and $h_{C_m}=h_{\rho_1}$. 
By Deligne-Mostow theory (see \cite[Proposition 4.1]{looijenga2007uniformization}), we have:
\begin{prop}
The Hermitian form $h_{C_m}$ has signature $(1, N-3)$.
\end{prop}
For a unitary Hermitian space $(T, h)$ over $\QQ(\zeta_3)$ we write
\begin{equation*}
\BB(T)\coloneqq \PP \{x\in T\otimes \CC \big{|} h(x, x)>0\}
\end{equation*} 
for the associated complex hyperbolic ball. The isomorphism type of the unitary Hermitian form $(T_{C_m}, h_{C_m})$ does not depend on the choice of $m$.

By \cite[Construction $3.2.1$]{rohde2009cyclic} we have a family $\calC\to (P^N)^\circ$ of Deligne-Mostow curves.
It is a branched covering of $P\times (P^N)^\circ$ with discriminant locus $D=\sum_{k=1}^{9}\alpha_k D_k$, where $D_k=\{(m(k),m) | m\in (P^N)^\circ \}\subset P\times (P^N)^\circ$.
The group $\FS_\alpha$ acts on $(P^N)^\circ$ freely in a natural way.
Denote by $\overline{D}_k$ the image of $D_k$ in $P\times ((P^N)^\circ/\FS_{\alpha})$.
For any $g\in \FS_{\alpha}$, $\overline{D}_k=\overline{D}_{g\cdot k}$.
Define $\overline{D}\coloneqq \sum_{k=1}^{N}\frac{\alpha_k}{|\FS_{\alpha}\cdot k|} \overline{D}_k$ to be a divisor of $P\times ((P^N)^\circ/\FS_{\alpha})$.
Let $\overline{\pi}\colon \overline{\calC}\to (P^N)^\circ/\FS_{\alpha}$ be the cyclic covering of $P\times ((P^N)^\circ/\FS_{\alpha})$ branched over $\overline{D}$. Then $\overline{\pi}$ is a family of Deligne-Mostow curves.
 
From now on we fix a base point $o$ on $(P^N)^\circ/\FS_{\alpha}$, and denote by $\BB_{DM} \coloneqq \BB(T_{C_o})$.
There is a sub-local system $\calT_{DM}$ of $R^1 \overline{\pi}_*\QQ(\zeta_d)$ whose stalk at $m\in (P^N)^\circ/\FS_{\alpha}$ is $T_{C_m}\subset H^1(C_m, \QQ(\zeta_d))$.
Let $\Gamma_{DM}$ be the monodromy group of $\calT_{DM}$ (with the base point $o$). 
Then $\Gamma_{DM}$ acts naturally on $T_{C_o}$, hence also on $\BB_{DM}$. 
From the above construction, we also have an analytic morphism $\Prd_{DM}\colon (P^N)^\circ/\FS_{\alpha}\to \Gamma_{DM}\bs \BB_{DM}$, which is called the period map.

We consider the action of $\SL(2, \CC)$ on $P^N$ together with the polarization $\La=\boxtimes_{i=1}^N\calO(2d\alpha_i)$.
From \cite[\S 4.1]{deligne1986monodromy}, the points in $(P^N)^\circ$ are stable. 
The $\FS_\alpha$-action on $(P^N)^\circ$ descends to the GIT-quotient $\SL(2)\dbs (P^N)^{\circ}$.
We define
\begin{equation}
\label{definition of GIT model of D-M}
\calF_{DM}\coloneqq (\SL(2)\dbs (P^N)^{\circ})/\FS_\alpha\cong \SL(2)\dbs ((P^N)^{\circ}/\FS_\alpha).
\end{equation}
The period map descends to $\calF_{DM}$, which we still denote by $\Prd_{DM}$.
\begin{thm}[\cite{deligne1986monodromy}, \cite{mostow1986generalized}]
\label{theorem: deligne-mostow}
	Assume that $(\alpha_1, \cdots, \alpha_N)$ satisfies \\
	\textup{($\Sigma$INT):} $0<\alpha_i<1$ for all $i$,  $\sum \alpha_i=2$ and for $1\le i<j\le N$ such that $\alpha_i+\alpha_j<1$, we have $(1-\alpha_i-\alpha_j)^{-1}$ is an integer if $\alpha_i\ne \alpha_j$, or a half-integer if $\alpha_i=\alpha_j$. \\
	Then the group $\Gamma_{DM}$ is a lattice, and the period map $\Prd_{DM}\colon \calF_{DM}\to \Gamma_{DM}\bs \BB_{DM}$ is an open embedding.
\end{thm}

Mostow \cite{mostow1988discontinuous} found an equivalent condition on the Deligne-Mostow data $\alpha=(\alpha_i)_{i\in [N]}$ for $\Gamma_{DM}$ to be a lattice in $O(T_{C_o}, h_{C_o})$, and gave a complete list of all such $\alpha$.

\subsection{Moduli of $3+6$ Points on $P$}
\label{subseciton: moduli of ordered points}
From now on we consider the case $\alpha=(\frac{1}{3}, \frac{1}{3}, \frac{1}{3}, \frac{1}{6}, \frac{1}{6},\frac{1}{6},\frac{1}{6},\frac{1}{6},\frac{1}{6})$.
There are natural projections $P^k\to \PP \calV_k$ sending a tuple of $k$ points on $P$ to a polynomial $X_0^{6-k} F_k$ such that $F_k$ vanishes on the tuple. 
This induces an isomorphism $p_{\calF}\colon P^9/\FS_{\alpha} = (P^3\times P^6)/\FS_{\alpha}\to \PP\calV_3\times \PP\calV_6$ where $\FS_{\alpha}=\mathfrak{S}_3 \times \mathfrak{S}_6$.
By straightforward calculation, we have 
\begin{equation*}
p_{\calF}^*(\calO_{\PP \calV_3}(2)\boxtimes\calO_{\PP \calV_6}(1) )\cong \calO_{P}(2)^{\boxtimes 3}\boxtimes \calO_{P}(1)^{\boxtimes 6}.
\end{equation*}

Therefore, we have an isomorphism between two GIT-quotients:
\begin{equation}
\label{equation: pF}
p_{\calF}\colon \SL(2, \CC) \dbs (P^9, \calO_{P}(2)^{\boxtimes 3}\boxtimes \calO_{P}(1)^{\boxtimes 6})/\FS_{\alpha} \to 
\SL(2, \CC)\dbs (\PP\calV_3\times \PP\calV_6, \calO_{\PP \calV_3}(2)\boxtimes\calO_{\PP \calV_6}(1))
\end{equation}

The left hand side in the morphism \eqref{equation: pF} is exactly the GIT-quotient in Deligne-Mostow theory, see the discussion in \S \ref{subsection: Review of Deligne-Mostow's Theory}.
From Deligne-Mostow theory, we have an algebraic open embedding $\Prd_{DM}\colon \calF_{DM}\to \Gamma_{DM}\bs \BB_{DM}$ (see Theorem \ref{theorem: deligne-mostow}).
From the isomorphism \eqref{equation: quotient by center}, we know that the right hand side of \eqref{equation: pF} is isomorphic to the GIT-quotient $G\dbs (\PP\calV, \calO(1))$. Recall that we have defined the moduli space $\calF_S=G\dbs \PP \calV^{\circ}$.
The restriction of the morphism $p_{\calF}$ to $\calF_{DM}$ gives rise to an isomorphism
\begin{equation*}
p_{\calF}\colon \calF_{DM}\to \calF_S.
\end{equation*}
At this point, we have the following diagram:
\begin{equation}
\label{partial diagram}
\begin{tikzcd}
\calF_{DM} \arrow{d}{p_{\calF}}\arrow{r}{\Prd_{DM}} & \Gamma_{DM} \bs \BB_{DM} \\
\calF_S
\end{tikzcd}
\end{equation}

In \S \ref{subsection: period map Prd_W}, we will define a period map $\Prd_S\colon \calF_S\to \Gamma_S\bs\BB_S$ and show its injectivity. In Theorem \ref{theorem: main}, Diagram \eqref{partial diagram} will be completed into Diagram \eqref{main diagram}.

\section{An Explicit Study of the Fibration Structure on $W_F$}
\label{section: trivialization}
Let $C$ be the $6$-fold Deligne-Mostow cover of $P$ associated with $\alpha=(\frac{1}{3}, \frac{1}{3}, \frac{1}{3}, \frac{1}{6}, \frac{1}{6}, \frac{1}{6}, \frac{1}{6}, \frac{1}{6}, \frac{1}{6})$.
Let $D$ be the smooth projective curve determined by the affine equation $u^2=v(v^3+1)$. The automorphism group $\Aut(D)$ is isomorphic to $\mu_6$.
We will show that the pullback of $W_F\to P$ to $C$ is birational to the product $C\times D$.
We will construct an action of $\Aut(D)$ on $C$ and show that the K3 surface $W_F$ is birational to the quotient of $C\times D$ by the diagonal action of $\Aut(D)$.

\subsection{Trivialization of the Isotrivial Fibration via Base-change}
\label{subsection: trivialization}

Given a sextic polynomial $F=X_0^3 F_3(X_1,X_2)+F_6(X_1,X_2)$, we have a natural morphism $W_F\to S_F\rightarrow \PP V$ and a fibration $\pi_F\colon W_F\rightarrow P$ as defined in \S \ref{section: Sextics and K3 Surfaces}.
Recall that $P$ can be naturally identified with the set of lines on $\PP V$ passing through $p$.

Define 
\[
P^\circ\coloneqq P-\{[0:1],[1:0] \}-Z(F_3)-Z(F_6), \quad W_F^\circ\coloneqq \pi_F^{-1}(P^{\circ})\cap \AAA_{s,x_1,t}^3.
\]
Here the definition of the indices $s, x_1, t$ can be found in the discussion after Proposition \ref{proposition: fibration}.
Recall that we denote $F_i(1,t)$ by $f_i(t)$.
The closure of $W_F^{\circ}$ in $\AAA_{s,x_1,t}^3$ is defined by
\begin{equation}
\label{equ: W_F}
s^2=x_1^3 f_3(t) + x_1^6 f_6(t),
\end{equation}

From Equation \eqref{equ: W_F}, we see that for any point $t\in P^\circ$, the curve $\pi^{-1}(t)\subset\AAA_{x_1, s}^2$ is determined by the equation
\begin{equation}\label{fiber}
\left(\frac{s}{x_1}\right)^2=x_1(f_3(t) + x_1^3 f_6(t)).
\end{equation}
Since $f_3(t), f_6(t)\ne 0$, the affine curve $\pi^{-1}(t)$ is isomorphic to the affine curve $D^{\circ}\coloneqq Z(u^2-v (v^3+1))\subset \AAA^2_{u,v}$. 

Let $C^{\circ}$ be the preimage of $P^{\circ}$ in the Deligne-Mostow curve $C$.
Then $C^{\circ}$ is determined by 
\begin{equation}
\label{equation: C}
y^6=f_3(t)^2 f_6(t).
\end{equation}
in the affine space $\AAA_{y,t}^2$. 

Put Equations \eqref{fiber} and (\ref{equation: C}) together we obtain
$$
(\frac{s}{x_1})^2 = x_1 f_3(t) + x_1^4 \frac{y^6}{f_3(t)^2},
$$
which can be rewritten as
$$
(\frac{sy}{x_1 f_3(t)})^2 = \frac{x_1 y^2}{f_3(t)} + (\frac{x_1 y^2}{f_3(t)})^4,
$$
and these can be viewed as identities of regular functions on the fiber product $W_F^{\circ}\times_{P^{\circ}}C^{\circ}\subset \AAA_{s,x_1,y,t}^4$.

An element in the fiber product $W_F^{\circ}\times_{P^{\circ}}C^{\circ}$ can be represented by a tuple $(s,x_1,y,t)$. Then the map 
\begin{equation}
\label{map: kappa}
W_F^{\circ}\times_{P^{\circ}}C^{\circ} \to \AAA^2_{u,v}\times \AAA^2_{y,t},\quad (s,x_1,y,t)\mapsto (u,v,y,t)=(\frac{sy}{x_1 f_3(t)}, \frac{x_1 y^2}{f_3(t)}, y, t)
\end{equation}
defines a morphism $\kappa\colon W_F^{\circ}\times_{P^{\circ}}C^{\circ}\to D^{\circ}\times C^{\circ}$.

\begin{prop}
\label{proposition: quotient -product sturcture}
The morphism $\kappa$ is an isomorphism. In particular it defines a birational map $W_F\times_P C\dashrightarrow D\times C$.
\end{prop}
\begin{proof}
Since $D\times C$ is smooth (hence normal), it suffices to show that $\kappa$ is a bijection.
For any $(s,x_1,y,t)$ and $(s',x'_1,y',t')$ in $W_F^{\circ}\times_{P^{\circ}}C^{\circ}$, if $\kappa(s,x_1,y,t)=\kappa(s',x'_1,y',t')$, then we immediately have $y=y'$ and $t=t'$.
From $\frac{x_1 y^2}{f_3(t)}=\frac{x'_1 y^2}{f_3(t)}$ we have $x_1=x'_1$, and then from $\frac{sy}{x_1 f_3(t)}=\frac{s'y}{x_1 f_3(t)}$ we have $s=s'$.
This proves the injectivity of $\kappa$.

The product $D^\circ\times C^\circ$ in $\AAA^2_{u,v}\times \AAA^2_{y,t}$ is cut out by $u^2=v(v^3+1)$ and $y^6=f_3(t)^2 f_6(t)$.
For any $(u,v,y,t)$ satisfies the above relations, we have $\kappa(\frac{uvf_3(t)^2}{y^3}, \frac{vf_3(t)}{y^2}, y, t)=(u,v,y,t)$.
This proves the surjectivity of $\kappa$. 
\end{proof}

\begin{rmk}\label{remark: geometric char of $C$}
We give a more geometric characterization of $C$. 
Each smooth fiber $\pi_F^{-1}(t)$ of the elliptic fibration $\pi_F\colon W_F\to P$ is isomorphic to $D$ (in $6$ ways). 
Such isomorphisms make up an $\Aut(D)$-cover of $P-\mathrm{disc}(\pi_F)$ and this extends to a normal $\Aut(D)$-cover $C'$ of $P$. 
By Proposition \ref{proposition: quotient -product sturcture}, we have a birational map $W_F\times_P C\dashrightarrow D\times C$. 
Therefore, a point in $C$ with image in $p\in P-\disc(\pi_F)$ gives rise to an isomorphism between the corresponding fiber of $\pi_F$ with $D$. 
We thus obtain a rational map $C\dashrightarrow C'$. 
From \eqref{map: kappa}, we see that two different points in $C$ with the same image $p\in P-\mathrm{disc}(\pi_F)$ define different isomorphisms from $\pi_F^{-1}(p)$ to $D$. 
Thus the rational map $C\dashrightarrow C'$ is injective outside the indeterminacy locus, which extends to an isomorphism.

Moreover, there is a rational fibration $C\times_{l_d}D\dashrightarrow D/\mu_6\cong \PP^1$ with generic fiber isomorphic to $C$.
We remark that there are generically injective morphisms from $C$ to $W_F$.
We take a point in $D$ with trivial $\mu_6$-stabilizer and consider its $\mu_6$-orbit.
For any $p\in P-\disc(\pi_F)$, there are $6$ isomorphisms between $D$ and $\pi_F^{-1}(p)$. They send $\mu_6\cdot p\subset D$ to the same $\mu_6$-orbit in the fiber $\pi_F^{-1}(p)$.
These $\mu_6$-orbits form a $6$-fold cover of $P-\disc(\pi_F)$.
Its closure in $W_F$ is a singular curve with normalization isomorphic to $C$.
\end{rmk}

\subsection{A birational identification of $W_F$}
\label{subsection: birational}
We define the $\mu_6$-actions on $C$ and $D$ by
\begin{equation}\label{action on C and D}
	\zeta_6\cdot\colon C\rightarrow C,\,\, (y,t)\mapsto (\zeta_6 y,t),
	\quad \zeta_6\cdot\colon D\rightarrow D,\,\, (u,v)\mapsto (\zeta_6 u,\zeta_3 v).
\end{equation}
Denote by $l_d$ the corresponding diagonal action of $\mu_6$ on $C\times D$, namely,
\begin{equation}\label{action: diagonal}
	l_d\colon \mu_6\times (C\times D)\to C\times D,\quad (\zeta_6, (y,t,u,v))\mapsto (\zeta_6 y,t,\zeta_6 u,\zeta_3 v).
\end{equation}
Denote by $C\times_{l_d} D$ the quotient of $C\times D$ by $l_d$.

From Proposition \ref{proposition: quotient -product sturcture}, we have the projection 
\begin{equation*}
	\widetilde{\psi}\colon C^{\circ}\times D^{\circ}\rightarrow W_F^{\circ},\quad (y,t,u,v)\mapsto (\frac{uvf_3(t)^2}{y^3}, \frac{vf_3(t)}{y^2},t).
\end{equation*}
For each $(y,t,u,v)\in C^{\circ}\times D^{\circ}$, we have
\begin{equation*}
	\widetilde{\psi}(l_d(\zeta_6, (y,t,u,v))) = \widetilde{\psi}(\zeta_6 y,t,\zeta_6 u,\zeta_3 v) = \widetilde{\psi}(y,t,u,v).
\end{equation*}
Hence the diagonal action $l_d$ on $C\times D$ coincides with the Deck transformations of $\widetilde{\psi}\colon C^{\circ}\times D^{\circ}\rightarrow W_F^{\circ}$, hence
\begin{prop}
\label{proposition: birational}
The map $\widetilde{\psi}$ factors through a birational map from $C\times_{l_d} D$ to $W_F$, which we denote by $\psi$.  
\end{prop}
We define a $\mu_6$-action on $W_F$ by
\begin{equation}\label{action: fiberwise}
\zeta_6\cdot\colon W_F\rightarrow W_F, \quad (s,x_1,t) \mapsto (-s,\zeta_3 x_1, t),
\end{equation}	
and define a $\mu_6$-action on $C\times D$ by 
\begin{equation}
\label{action: sigma'_F}
\zeta_6\cdot \colon C\times D\rightarrow C\times D, \quad (y,t,u,v) \mapsto (y,t,\zeta_6 u,\zeta_3 v).
\end{equation}
This action descents to a $\mu_6$-action on $C\times_{l_d} D$:
\begin{equation*}
\zeta_6\cdot \colon C\times_{l_d} D \rightarrow C\times_{l_d} D, \quad  [(y,t,u,v)] \mapsto [(y,t,\zeta_6 u,\zeta_3 v)].
\end{equation*}

\begin{prop}
	\label{proposition: sigma_F is fiberwise}
	The actions of $\mu_6$ on both sides of the birational map $\psi\colon C\times_{l_d}D\dashrightarrow W_F$ are compatible. Moreover, the action of $\mu_6$ on $W_F$ preserves every fiber of $W_F\to P$.
\end{prop}
\begin{proof}
    We need to verify that
    \begin{equation}
    \label{equation: sigma}
    	\psi(\zeta_6([(y,t,u,v)])) = \zeta_6(\psi([(y,t,u,v)]))
    \end{equation}
    holds for all $[(y,t,u,v)]\in C^{\circ}\times D^{\circ}$.
    We have 
\begin{equation*}
\psi(\zeta_6([(y,t,u,v)])) = \psi([(y,t,,\zeta_6 u,\zeta_3 v)]) = (-\frac{uvf_3(t)^2}{y^3}, \zeta_3\frac{vf_3(t)}{y^2},t)
\end{equation*} 
and 
\begin{equation*}
\zeta_6(\psi([(y,t,u,v)])) = \zeta_6((\frac{uvf_3(t)^2}{y^3}, \frac{vf_3(t)}{y^2},t)) = (-\frac{uvf_3(t)^2}{y^3}, \zeta_3\frac{vf_3(t)}{y^2},t),
\end{equation*}
hence the Equation \eqref{equation: sigma} holds.

	The fibration $C\times_{l_d} D\rightarrow C/\mu_6\cong P$ is induced by the projection $C\times D\rightarrow C$ after taking quotient by the $\mu_6$-actions.
	Hence $[(y,t,u,v)]$ and $[(y,t,\zeta_6 u,\zeta_3 v)]$ have the same image in $P$. Thus the action of $\mu_6$ on $W_F$ preserves every fiber.
\end{proof}

\section{Period Map for $W_F$}
\label{section: period map by K3}
In this section we characterize the period domain and the period map for the $K3$ surfaces $W_F$ and prove the injectivity of the period map. Combining this with a dimension counting we show that the period map is an open embedding into an arithmetic ball quotient of dimension $6$.

\subsection{A Natural Lattice associated with $W_F$}
Shioda \cite[Theorem $1.1$]{shioda1972elliptic} calculated the N\'eron-Severi group for any elliptic surface with a section using the geometric information of the singular fibers and sections. 
In our case, for $F\in \calV^{\circ}$ (see \S \ref{subseciton: GIT moduli of type I}), the $K3$ surface $W_F$ has an elliptic fibration (with a natural section) over $P$.
This fibration has nine singular fibers, with three of type \uppercase\expandafter{\romannumeral4} and other six of type \uppercase\expandafter{\romannumeral2} (see Proposition \ref{proposition: sigular fiber type}). 
Let $P_F$ be the subgroup of the Picard group $\Pic(W_F)$  generated by the section, a smooth fiber and the irreducible components of the three fibers of type \uppercase\expandafter{\romannumeral4} which do not intersect with the section. 
Let $Q_F$ be the orthogonal complement of $P_F$ in $H^2(W_F, \ZZ)$. 
The isomorphism type of the pair $(P_F, Q_F)$ does not depend on the choice of $F$ in $\calV^{\circ}$. 

\begin{prop}
The inclusion $P_F\subset \Pic(W_F)$ is primitive. 
\end{prop}
\begin{proof}
By Shioda \cite[Theorem $1.1$]{shioda1972elliptic}, the quotient $\Pic(W_F)/P_F$ is isomorphic to the Mordell-Weil group $\MW(W_F)$. 
The torsion part $\MW(W_F)_{tor}$ is a subgroup of the smooth part of every fiber, see \cite[Remark $1.10$]{shioda1972elliptic}. 
Since the fibration on $W_F$ has singular fibers of type \uppercase\expandafter{\romannumeral2}, we conclude that $\MW(W_F)$ is torsion-free. Thus $P_F$ is primitive in $\Pic(W_F)$.
\end{proof}

We characterize $(P_F, Q_F)$ in the next proposition.

\begin{prop}
\label{proposition: shioda}
For any $F\in \calV^{\circ}$, we have $P_F\cong U\oplus A_2(-1)^3$ and $Q_F \cong A_2\oplus E_6(-1)^2$. 
Here $U$ represents for the hyperbolic lattice of rank two, and $A_n, D_n, E_n$ represent for the positive definite root lattices associated to the correponding Dykin diagrams. 
\end{prop}
\begin{proof}
The first isomorphism $P_F\cong U\oplus A_2(-1)^3$ directly follows from its definition.
Since the discriminant forms of $U\oplus A_2(-1)^3$ and $A_2\oplus E_6(-1)^2$ are inverse to each other, there exists primitive embedding $U\oplus A_2(-1)^3\hookrightarrow \Lambda_{K3}=U^3\oplus E_8(-1)^3$ with the orthogonal complement isomorphic to $A_2\oplus E_6(-1)^2$.
	By Nikulin \cite[Theorem 1.14.4]{nikulin1979integer}, a primitive embedding of $P_F$ into the $K3$ lattice $\Lambda_{K3}$ is unique up to automorphisms of $\Lambda_{K3}$. This implies that $Q_F\cong A_2\oplus E_6(-1)^2$. 
\end{proof}

\begin{rmk} 
For generic $F$, the lattice $P_F$ is actually the Picard lattice. This will be clear after we show the injectivity of the period map. See Corollary \ref{proposition: generic Pic and Tr}.
\end{rmk}

\subsection{The Period Map for $W_F$}
\label{subsection: period map Prd_W}
Recall that we have defined an action (see \eqref{action: fiberwise}) of $\mu_6$ on the $K3$ surface $W_F$. 
It is clear that the induced action of $\mu_3\subset \mu_6$ on $H^2(W_F, \ZZ)$ fixes the classes of the section and the irreducible components of the three singular fibers of type \uppercase\expandafter{\romannumeral4} in $W_F\to P$. 
We denote by $H^2(W_F, \ZZ)^{\mu_3}$ the $\mu_3$-invariant sublattice of $H^2(W_F, \ZZ)$.
It then follows from the definition of $P_F$ that $P_F$ is contained in $H^2(W_F, \ZZ)^{\mu_3}$.
In particular, $\mu_3$ preserves $Q_F$.

Since the quotient $W_F/\mu_3$ is a rational surface, the $\mu_3$-action on $W_F$ is non-symplectic (it is well-known that the quotient of a $K3$ surface by a finite symplectic automorphism is a $K3$ surface with singularities). 
Therefore, the invariant sublattice $H^2(W_F, \ZZ)^{\mu_3}$ is orthogonal to $H^{2,0}$. 
This implies that $H^2(W_F, \ZZ)^{\mu_3}$ is a primitive sublattice of the Picard lattice $\Pic(W_F)$.
We now have the inclusions:
\begin{equation}
\label{equation: PF}
P_F\subset H^2(W_F, \ZZ)^{\mu_3}\subset \Pic(W_F).
\end{equation}

\begin{defn}
Let $T_F$ be the $\mu_3$-characteristic subspace of $H^2(W_F, \QQ(\zeta_3))$ such that $(T_F)_{\CC}\supset H^{2,0}(W_F)$. 
\end{defn}
Let $\epsilon$ be the intersection form on $H^2(W_F, \ZZ)$. 
We have a Hermitian form
\begin{equation*}
h_{\epsilon}\colon T_F\times T_F\to \CC, h_{\epsilon}(x,y)=\epsilon(x, \overline{y})
\end{equation*}
which has signature $(1, *)$.
Now we fix arbitrarily an element $F_o\in \calV^{\circ}$ as our base point. 
Let $\BB_S=\BB(T_{F_o})$ be the complex hyperbolic ball associated with $(T_{F_o}, h_{F_o})$. Define 
\begin{equation*}
\Gamma_S\coloneqq \{g\in O(Q_F) \big{|} g\circ\zeta_3=\zeta_3\circ g\},
\end{equation*}
where the element $\zeta_3\in \mu_3$ acts on $Q_F$. 

Next we define a period map $\Prd_S\colon \calF_S \to \Gamma_S\bs\BB_S$. 
For $F\in \calV^{\circ}$, take a path $\gamma$ in $\calV^{\circ}$ from $F_o$ to $F$.
It induces an isomorphism $\gamma^*\colon T_F\to T_{F_o}$.
The line $\gamma^*(T^{2,0}_F)$ represents a point in $\BB_S$. 
We define $\Prd_S(F)$ to be $[\gamma^*(T^{2,0}_F)]\in \Gamma_S \bs \BB_S$.
This is well-defined since for two choices of $\gamma$, the corresponding points in $\BB_S$ are in one orbit of $\Gamma_S$.

Therefore, we have an analytic morphism $\Prd_S\colon \PP\calV^{\circ}\to \Gamma_S\bs\BB_S$, which is constant on every $G_S$-orbit.
Thus $\Prd_S$ descends to 
\begin{equation}
\label{equation: period for sextic}
\Prd_S\colon \calF_S\to \Gamma_S\bs\BB_S
\end{equation}
which we call the period map for the sextic curves $Z(F)$.

\begin{prop}
\label{proposition: open embedding}
	The period map $\Prd_S\colon \calF_S\to \Gamma_S\bs \BB_S$ is injective.
\end{prop}
\begin{proof}
	Suppose there are two sextic curves $Z(F), Z(F')\in \PP \calV^{\circ}$ such that $\Prd_S(F)=\Prd_S(F')$. Then we can choose paths $\gamma, \gamma'$ (respectively) in $\calV^{\circ}$ connecting $F, F'$ (respectively) to $F_o$, such that $\gamma^*(H^{2,0}(W_F))$ and $\gamma'^*(H^{2,0}(W_{F'}))$ lie in one $\Gamma_S$-orbit. Take $g\in \Gamma_S$ such that $g(\gamma^*(H^{2,0}(W_F)))=\gamma'^*(H^{2,0}(W_
{F'}))$. Let $\iota_Q\coloneqq \gamma'^{*-1} g \gamma^*$. Then $\iota_Q \colon (Q_F, \mu_3)\cong (Q_{F'}, \mu_3)$ such that $\iota_Q(T^{2,0}_{F})=T^{2,0}_{F'}$. 
	We aim to show $Z(F)\cong Z(F')$. 
	
	Let $f_F$ and $s_F$ be the classes of the fiber and the section of the elliptic fibration $W_F\to P$. Let $e_F=s_F+f_F$. 
The pair $(e_F, f_F)$ form a standard basis of the hyperbolic lattice $U$ (namely, $e_F^2=f_F^2=0, (e_F,f_F)=1$).

By Remark \ref{proposition: sigular fiber type}, each of the three singular fibers of $W_F\to P$ of type IV has three irreducible components. Let $x_i, y_i, z_i$ ($i=1,2,3$) be the classes of the irreducible components of a fiber of type IV, such that $x_i\cdot s_F=y_i\cdot s_F=0$ and $z_i\cdot s_F=1$. By Proposition \ref{proposition: shioda}, we have 
\begin{equation*}
P_F=\langle e_F, f_F\rangle\oplus \langle x_1, y_1\rangle \oplus \langle x_2, y_2\rangle \oplus \langle x_3, y_3\rangle \cong U\oplus A_2(-1)^3
\end{equation*}
and $Q_F\cong A_2\oplus E_6(-1)^2$. Similarly, we have $s_{F'}, f_{F'}, e_{F'}, x_i'$ and $y_i'$ for $F'$.

The two tuples $(P_F, e_F, f_F)$ and $(P_{F'}, e_{F'}, f_{F'})$ are isomorphic.
For a lattice $L$ we denote by $A_L\coloneqq L^{*}/L$ its discriminant group.
Since $A_{A_2}\cong \ZZ/3$ and $O(A_2^3)\to O(A_{A_2}^3)$ is surjective, we can choose $\iota_P\colon(P_F, e_F, f_F)\cong (P_{F'}, e_{F'}, f_{F'})$ such that $\iota_P^*\colon A_{P_F}\cong A_{P_{F'}}$ coincides with $\iota_Q^*\colon A_{Q_F}\cong A_{Q_{F'}}$. Then we can glue  $\iota_Q$ and $\iota_P$ to 
\begin{equation*}
\iota\colon H^2(W_F, \ZZ)\cong H^2(W_{F'}, \ZZ)
\end{equation*}
which is a Hodge isometry. 

If $L$ is a lattice and $\alpha\in L$ is such that $\alpha\cdot \alpha=-2$, then the reflection $r_\alpha: x\mapsto x+(x\cdot\alpha)\alpha$ acts trivially on $A_L$ as it takes any $y\in L^*$ to $y+(y\cdot\alpha)\alpha\in y+L$. 
In particular, for a root $\alpha \in A_2$, the reflection $r_{\alpha}$ induces trivial action on the discriminant group of $A_2$. 
Thus we can suitably adjust the isomorphism $\iota_P\colon P_F\cong P_{F'}$, such that $\iota$ sends an effective root in $A_2$ to an effective root. 
We take $h\coloneqq 3e_F+4f_F-x_1-y_1-x_2-y_2-x_3-y_3=3 s_F + 4 f_F + z_1 + z_2 + z_3$, which is an effective element in $P_F$. We can define $h'$ similarly. Then $\iota(h)=h'$. We claim that both $h$ and $h'$ are ample. By Nakai-Moishezon criterion, an element in the Picard group of a complex $K3$ surface is ample if and only if it has positive self-intersection and positive intersection with every irreducible curve.
	We have $h\cdot s_F = 1, h\cdot f_F = 3, h\cdot x_i=h\cdot y_i=h\cdot z_i = 1$ and $h^2=18$. Suppose $C$ is an irreducible curve on $W_F$ such that $[C]\ne s_F, f_F, x_i, y_i, z_i$. Then $C$ is not contained in any fiber, hence $f_F\cdot C>0$, which implies that $h\cdot C> 0$.
Thus the class $h$ is ample, so is $h'$. 
Therefore, $\iota$ sends an ample class to ample class.  
By global Torelli theorem (\cite[Theorem $1$]{rapoport1975torelli}), there exists an isomorphism $\eta\colon W_F\cong W_{F'}$ which induces $\iota$. Note that $\iota$ is compatible with the actions of $\mu_3$.
From the faithfulness of the action of an automorphism of a $K3$ surface on the middle cohomology, we conclude that $\eta$ is compatible with the actions of $\mu_3$.
	
	We claim $\eta$ gives rise to an isomorphism between $Z(F)$ and $Z(F')$. 
	Recall that we denote by $\pi_F$ the elliptic fibration $W_F\to P$, see \S \ref{section: Sextics and K3 Surfaces}.
	Since $\iota(f_F)=f_{F'}$, the isomorphism $\eta$ maps each fiber of $\pi_F$ to a fiber of $\pi_{F'}$.
	Hence $\eta$ also maps a singular fiber to a singular fiber of the same type.
	Note that the base $P$ can be identified with the subvariety $s_F(P)$ of $W_F$, where $s_F$ denotes the section of $\pi_F$ and $[s_F(P)]=e_F-f_F$ in $P_F$ (same thing holds for $F'$).
	Since $\iota(e_F-f_F)=e_{F'}-f_{F'}$, $\eta$ maps $s_F(P)$ isomorphically to $s_{F'}(P)$.
	Thus $\eta$ sends $Z(F_3), Z(F_6)$ to $Z(F_3'), Z(F_6')$ respectively.
	Hence it induces an element in $G_S$ that identifies $Z(F)$ with $Z(F')$.
\end{proof}

\begin{cor}
	\label{proposition: generic Pic and Tr}
	The period map $\Prd_S$ is an open embedding.
	For a generic $F\in \calV^{\circ}$, we have $\Pic(W_F)=P_F$ and the transcendental lattice $T(W_F)=Q_F$.
\end{cor}
\begin{proof}
	By the injectivity of the period map $\Prd_S$, we have $\dim \BB_S\ge 6$, hence $\dim T_F\ge 7$. 
	Since $\rank(Q_F)=14$, we must have $\dim T_F = 7$ and $\dim \BB_S = 6$.
	The equality $\dim \calF_S=\dim \BB_S=6$ and the injectivity imply that $\Prd_S$ is an open embedding.
	Since $\dim T_F=7$ and $\rank(Q_F)=14$, we have $(Q_F)_{\QQ(\zeta_3)}= T_F\oplus \overline{T}_F$.
	A generic element in $\BB_S$ is not orthogonal to any elements in $Q_F$.
    By the openness of $\Prd_S$, we know that for a generic choice of $F$, the $K3$ surface $W_F$ satisfies $\Pic(W_F)\cap Q_F=0$. 
    Therefore, we have $\Pic(W_F)=P_F$ and $T(W_F)=Q_F$.
\end{proof}

\section{Hodge Structures of $W_F$ from Deligne-Mostow theory}
\label{section: hodge structure}
In this section we establish a relation (see Proposition \ref{proposition: isomorphism between hodge structures} and \ref{proposition: characterization of T_F}) between the weight-two Hodge structures of $K3$ surfaces $W_F$ and the weight-one Hodge structures of Deligne-Mostow curves $C_m$. This relation is obtained from the explicit birational model for $W_F$ (see Proposition \ref{proposition: birational}) in \S \ref{section: trivialization} combining with the Chevalley-Weil formula that will be introduced in \S \ref{subsection: c-w formula}.

\subsection{Chevalley-Weil Formula}
\label{subsection: c-w formula}
In this section we introduce the Chevalley-Weil formula.
See \cite{chevalley1934uber}, \cite{naeff2005chevalley}.
Let $f\colon X\rightarrow Y$ be a Galois covering of degree $n$ between two nonsingular projective curves $X$, $Y$ over $\CC$, with $G\coloneqq \Gal(X/ Y)$ the Galois group. 
The action of $G$ on $X$ induces a linear representation $G\rightarrow \GL(H^0(X,\Omega_X^1))$.
Chevalley and Weil \cite{chevalley1934uber} studied the multiplicity of a given irreducible representation of $G$ in $H^0(X,\Omega_X^1)$.

Let $p_1,\cdots, p_r\in Y$ be the branched points of $f$.
Let $e_i$ be the ramification index of the points in $f^{-1}(p_i)$. For a point $q\in f^{-1}(p_i)$, we denote by $g_q$ the element in $G$ such that $g_q$ fixes $q$ and the pullback of $g_q$ on the cotangent space $T^*_q(X)$ is by multiplying $\zeta_{e_i}$. 

For an irreducible representation $\rho\colon G\to \GL(d_{\rho}, \CC)$ with character $\chi_{\rho}$ and degree $d_{\rho}$. Let $m_{\rho}$ be the multiplicity of $\rho$ in the representation $G\rightarrow \GL(H^0(X,\Omega_X^1))$. We denote by $N_{ij}$ the multiplicity of $\zeta_{e_i}^j$ as an eigenvalue of the matrix $\rho(g_q)$ for $q\in f^{-1}(p_i)$. The number $N_{ij}$ is well-defined since for another $q'\in f^{-1}(p_i)$, the element $g_{q'}$ is conjugate to $g_q$ in $G$. We write $\left\langle r \right\rangle=r-\lfloor r \rfloor$ for the fractional part of a rational number $r$.
The following is the Chevalley-Weil formula.
\begin{thm}[Chevalley-Weil]
\label{theorem: C-W}
Let $\rho$ be an irreducible summand of the representation $G\rightarrow \GL(H^0(X,\Omega_X^1))$. Then
\begin{equation*}
m_{\rho} = d_{\rho}(g(Y)-1) + \delta + \sum_{i=1}^{r} \sum_{j=0}^{e_i-1} N_{ij}\left\langle -\frac{j}{e_i} \right\rangle,
\end{equation*}
where $\delta=1$ if $\chi$ is the trivial character and $\delta=0$ otherwise, and $g(Y)$ denotes the genus of $Y$. 
\end{thm}

If $G$ is abelian, then the element $g_q$ does not depend on the choice of $q\in f^{-1}(p_i)$. We then call $g_q$ the local monodromy of $f$ at $p_i$.

In our case, the curve $C$ is defined by the affine equation (\ref{equation: C}) and the map $f\colon C\to P$ is a Galois cyclic cover of degree $6$ with $\Gal(C/P)=\mu_6$ such that the element $\zeta_6\in \mu_6$ sends $(t,y)$ to $(t, \zeta_6 y)$. The Deligne-Mostow data of $C$ is $(\frac{1}{3},\frac{1}{3},\frac{1}{3},\frac{1}{6},\frac{1}{6},\frac{1}{6},\frac{1}{6},\frac{1}{6},\frac{1}{6})$. 
Denote by $p_1, p_2, p_3$ the branched points with weight $\frac{1}{3}$, and $p_4, \cdots, p_9$ the branched points with weight $\frac{1}{6}$.
The local monodromy at $p_i$ equals to $\zeta_3\in \mu_6$ (when $1\le i\le 3$) or $\zeta_6\in \mu_6$ (when $4\le i\le 9$).
Define characters
\begin{equation}
\label{equation: rho_k}
\rho_k\colon \mu_6\rightarrow \CC^*, \rho_k(\zeta_6)=\zeta_6^k, 0\le k \le 5.
\end{equation}
Then for $\rho_k$ we have $N_{ij}=1$ if $j=k$, and $N_{ij}=0$ otherwise. Applying Theorem \ref{theorem: C-W} for $\rho_k$, we have
\begin{equation}
\label{equation: C-W}
m_{\rho_k} = -1 + \delta + 3\left\langle -\frac{k}{3} \right\rangle + 6\left\langle -\frac{k}{6} \right\rangle,
\end{equation}
where $\delta=1$ if $k=0$, and $\delta=0$ otherwise. More explicitly, we have $$(m_{\rho_0}, m_{\rho_1}, m_{\rho_2}, m_{\rho_3}, m_{\rho_4}, m_{\rho_5})=(0, 6, 4, 2, 3, 1).$$ 
The sum of $m_{\rho_k}$ is $16$, which is the genus of the curve $C$.
We will use the value of $m_{\rho_1}$ and $m_{\rho_5}$ to prove Lemma \ref{lemma: dim T_C^{1,0}}.

\subsection{Hodge Structures of $W_F$}

We analyze the Hodge structure of $W_F$ through the birational isomorphism (see Proposition \ref{proposition: birational}) between $W_F$ and $C\times_{l_d} D$.
For the definition of $l_d$, see (\ref{action: diagonal}). 

If two complex smooth projective surfaces are birational to each other, then they have canonically isomorphic transcendental lattices, see \cite[Lemma $3.1$]{shioda2008k3}. Therefore, we can define the transcendental lattice for a complex irreducible surface (which is not required to be smooth or projective) to be the transcendental lattice of the minimal model of any of its projective compactifications.
From Proposition \ref{proposition: birational} we have a birational map $C\times_{l_d} D\dashrightarrow W_F$. 

There are $54$ points on $C\times D$ with nontrivial stabilizers under the diagonal action $l_d$.
The quotient-product surface $C\times_{l_d} D$ has $27$ cyclic quotient singularities. 
Recall that every cyclic quotient singularity is locally analytically isomorphic to the quotient of $\CC^2$ by the action of a diagonal linear automorphism with eigenvalues $\zeta_n$ and $\zeta_n^q$ with $\gcd(n,q)=1$.
This is called a singularity of type $\frac{1}{n}(1,q)$ (see \cite[Remark $1.1$]{bauer2012classification}).
A cyclic quotient singularity can be resolved by the so-called Hirzebruch-Jung strings (see \cite[Chapter \uppercase\expandafter{\romannumeral3}, \S 5]{barth2004compact}).
In our case, there are $6$ ($15$, $6$ resp.) of them of type $\frac{1}{6}(1,1)$ ($\frac{1}{3}(1,1)$, $\frac{1}{2}(1,1)$ resp.). 
The minimal resolution $\widetilde{C \times_{l_d} D}$ of $C\times_{l_d} D$ is obtained by blowing up each singularity once.
For a singularity of type $\frac{1}{6}(1,1)$ ($\frac{1}{3}(1,1)$, $\frac{1}{2}(1,1)$ resp.), we obtain an exceptional curve with self-intersection $-6$ ($-3$, $-2$ resp.).

Let $\widetilde{C \times D}$ be the blowup of $C\times D$ at the $54$ points.
Then we have $\widetilde{C\times D}/{l_d}\cong \widetilde{C \times_{l_d} D}$. 
We then have $H^2(\widetilde{C\times D}/{l_d},\QQ)\cong H^2(\widetilde{C\times D},\QQ)^{l_d}$ and $T(\widetilde{C\times D}/{l_d})_{\QQ}\cong T(\widetilde{C\times D})^{l_d}_{\QQ}$. 
By the natural birational morphisms $\widetilde{C\times D}/{l_d}\to W_F$ and $\widetilde{C\times D}\to C\times D$, we have isomorphisms $T(W_F)\cong T(\widetilde{C\times D}/{l_d})$ and $T(C\times D)\cong T(\widetilde{C\times D})$ between transcendental lattices. 
Therefore, we have an isomorphism $T(W_F)_{\QQ}\cong T(C\times D)^{l_d}_{\QQ}$.

The $\QQ$-vector space $T(W_F)_{\QQ}\cong T(C\times D)^{l_d}_{\QQ}$ is a subspace of 
\begin{equation}
\label{equation: kunneth}
H^2(C\times D,\QQ)^{l_d}= (H^2(C,\QQ)\otimes H^0(D,\QQ))\oplus (H^0(C,\QQ)\otimes H^2(D,\QQ))\oplus (H^1(C,\QQ)\otimes H^1(D,\QQ))^{l_d}.
\end{equation}
The first two summands of the right part of Equation \eqref{equation: kunneth} belong to $\Pic(C\times D)_{\QQ}$.
Thus we have an injective map
\begin{equation}
\label{equation: inclusion of T}
\lambda\colon T(W_F)_{\QQ}\hookrightarrow (H^1(C,\QQ)\otimes H^1(D,\QQ))^{l_d},
\end{equation}
which is indeed an isomorphism as we will show in Proposition \ref{proposition: isomorphism between hodge structures}.

Let $\rho_D$ and $\overline{\rho}_D$ be the characters of $\mu_6$ corresponding to the induced actions on $H^{0,1}(D)$ and $H^{1,0}(D)$.
By definition we have $H^{1,0}(D)=H^1(D, \CC)_{\overline{\rho}_D}$ and $H^{0,1}(D)=H^1(D, \CC)_{\rho_D}$. 
\begin{lem}\label{lemma: description of rho_D}
We have $\rho_D=\rho_5$.
\end{lem}
\begin{proof}
	Recall that $D$ is determined by the equation $u^2=v(v^3+1)$, and $\frac{dv}{u}$ represents a generator of $H^{1,0}(D)$. 
	We have defined the action of $\zeta_6\in\mu_6$ on $D$ by sending $(u,v)$ to $(\zeta_6 u, \zeta_3 v)$. 
	Thus $\zeta_6\in \mu_6$ acts on $\frac{dv}{u}$ by multiplying $\zeta_6$.
	Hence $\overline{\rho}_D=\rho_1$ and $\rho_D=\rho_5$, with $\rho_k$ defined in \eqref{equation: rho_k}.
\end{proof}

\begin{defn}We denote by $T_C$ the $\rho_D$-characteristic subspace of $H^1(C, \QQ(\zeta_3))$, and let $T_C\otimes \CC=T_C^{1,0}\oplus T_C^{0,1}$ be the Hodge decomposition.
\end{defn}

\begin{lem}
\label{lemma: dim T_C^{1,0}}
We have $\dim(T_C^{1,0})=1$ and $\dim(T_C^{0,1})=6$.
\end{lem}
\begin{proof}
	By Lemma \ref{lemma: description of rho_D}, we have $\rho_D=\rho_5$.
	By Equation \eqref{equation: C-W}, we have 
	\begin{equation*}
	\dim T^{1,0}_C=m_{\rho_5}=1, \dim (\overline{T}_C)^{1,0}=m_{\rho_1}=6.
	\end{equation*}
	Thus we have $\dim T^{0,1}_C=\dim (\overline{T}_C)^{1,0}=6$ and $\dim (\overline{T}_C)^{0,1}=\dim T^{1,0}_C=1$. 
\end{proof}

The next proposition relates the polarized weight two Hodge structure on $Q_F$ and the polarized weight one Hodge structure on $H^1(C, \QQ)$.

\begin{prop}
\label{proposition: isomorphism between hodge structures}
For $F=X_0^3 F_3+F_6\in \calV^{\circ}$, we have an isomorphism
\begin{equation}
\label{equation: delta}
(Q_F)_{\QQ}\cong (H^1(C, \QQ)\otimes H^1(D, \QQ))^{l_d}.
\end{equation} 
After tensoring with $\CC$, this isomorphism sends $H^{2,0}(W_F)$ to $T_C^{1,0} \otimes H^{1,0}(D)$.
\end{prop}

\begin{proof}
The curve $D$ has genus $1$, and we have explicitly described (see Equation \eqref{action on C and D}) the action of $\mu_6$ on $D$. The vector space $(H^1(C,\QQ)\otimes H^1(D,\QQ))^{l_d} \otimes\CC$ can be decomposed as the direct sum of 
\begin{align}
\label{equation: invariant part}
T^{1,0}_C\otimes H^1(D)_{\overline{\rho}_D}, T^{0,1}_C\otimes H^1(D)_{\overline{\rho}_D}, H^{1,0}(C)_{\overline{\rho}_D} \otimes H^1(D)_{\rho_D}, H^{0,1}(C)_{\overline{\rho}_D} \otimes H^1(D)_{\rho_D}.
\end{align}
From \eqref{equation: invariant part} and Lemma \ref{lemma: dim T_C^{1,0}}, the dimension of $(H^1(C)\otimes H^1(D))^{l_d}$ equals to $14$.

If $F$ is generic, then by Proposition \ref{proposition: generic Pic and Tr}, we have $T(W_F)\cong Q_F$ and $\rank(T(W_F))=\rank(Q_F)=14$.
Therefore, the inclusion \eqref{equation: inclusion of T} is indeed an isomorphism 
\begin{equation*}
\lambda\colon (Q_F)_{\QQ}\cong (H^1(C,\QQ)\otimes H^1(D,\QQ))^{l_d}.
\end{equation*}
This isomorphism (after tensor with $\CC$) identifies $H^{2,0}(W_F)$ with $(H^{1,0}(C)\otimes H^{1,0}(D))^{l_d}=T^{1,0}_C\otimes H^{1,0}(D)$. 
We conclude Proposition \ref{proposition: isomorphism between hodge structures} for a generic choice of $F$.

Notice that both two sides of the isomorphism \eqref{equation: delta} are topologically defined, hence it still holds after deforming $F$. Therefore, the isomorphism holds for any $F\in \calV^{\circ}$.
\end{proof}

\subsection{Identification between Unitary Hermitian Forms}
\label{subsection: unitary hermitian form}
Recall that we have a birational map $C\times_{l_d} D\dashrightarrow W_F$ with compatible actions of $\mu_6$ on $C\times_{l_d} D$ and $W_F$, see Proposition \ref{proposition: sigma_F is fiberwise} and Equations \eqref{action: fiberwise}, \eqref{action: sigma'_F}.
Recall $\rho_D=\rho_5$ is the character of $\mu_6$ induced from its action on $H^{0,1}(D)$ (see Lemma \ref{lemma: description of rho_D}). 
Recall that $T_C$ is the $\rho_5$-characteristic subspace of $H^1(C, \QQ(\zeta_3))$, and $T_F$ is the $\mu_3$-characteristic subspace of $H^2(W_F, \QQ(\zeta_3))$ such that $(T_F)_{\CC}\supset H^{2,0}(W_F)$.
By \eqref{equation: PF}, we know that $T_F$ is a subspace of $(Q_F)_{\QQ(\zeta_3)}$.
The following proposition describes $T_F$.

\begin{prop}
\label{proposition: characterization of T_F}
For $F\in\calV^{\circ}$, we have $\lambda(T_F)= T_C\otimes H^1(D, \QQ(\zeta_3))_{\overline{\rho}_D}$, $(Q_F)_{\QQ(\zeta_3)} = T_F\oplus \overline{T}_F$ and $(P_F)_{\QQ(\zeta_3)}\cong H^2(W_F,\QQ(\zeta_3))^{\mu_3}$.
\end{prop}

\begin{proof}
From Proposition \ref{proposition: sigma_F is fiberwise} and \ref{proposition: isomorphism between hodge structures}, the actions of $\mu_6$ on $Q_F$, $(H^1(C)\otimes H^1(D))^{l_d}(\cong T(C\times D)^{l_d})$ are compatible with the isomorphism $\lambda\colon (Q_F)_{\QQ}\cong (H^1(C,\QQ)\otimes H^1(D,\QQ))^{l_d}$.

Note that $T_F$ is the $\mu_3$-characteristic subspace of $(Q_F)_{\QQ(\zeta_3)}$ such that $(T_F)_{\CC}\supset H^{2,0}(W_F)$. We have:
\begin{equation*}
(H^1(C,\QQ(\zeta_3))\otimes H^1(D, \QQ(\zeta_3))_{\overline{\rho}_D}) \cap ((H^1(C, \QQ(\zeta_3))\otimes H^1(D, \QQ(\zeta_3)))^{l_d})= T_C\otimes H^1(D, \QQ(\zeta_3))_{\overline{\rho}_D},
\end{equation*}
hence $T_C\otimes H^1(D, \QQ(\zeta_3))_{\overline{\rho}_D}$ is the $\mu_6$(also $\mu_3$)-characteristic subspace of $(H^1(C)\otimes H^1(D))^{l_d}$ that contains $T^{1,0}_C \otimes H^{1,0}(D)$. Therefore,
\begin{equation*}
\lambda(T_F)= T_C\otimes H^1(D, \QQ(\zeta_3))_{\overline{\rho}_D}. 
\end{equation*}
%
%
This identification implies that $\dim(T_F)=\dim(T_C)=7$.
The spaces $T_F$ and $\overline{T}_F$ are two different $\mu_3$-characteristic subspaces of $(Q_F)_{\QQ(\zeta_3)}$. Since $\dim(Q_F)=14$, we conclude $(Q_F)_{\QQ(\zeta_3)}= T_F\oplus \overline{T}_F$. Then $Q_F$ is fixed-point free under the action of $\mu_3$, thus $(P_F)_{\QQ(\zeta_3)}\cong H^2(W_F,\QQ(\zeta_3))^{\mu_3}$.
\end{proof}

Let $\varphi$ be the topological intersection form on $T(C\times D)$, which is symmetric.
Let $h_{\varphi}(x,y)=\varphi(x,\overline{y})$ be the corresponding Hermitian form. 
Let $\nu$ and $\xi$ be the topological symplectic forms on $H^1(C, \CC)$ and $H^1(D, \CC)$ respectively.
Let $h_{\nu}(x,y)=\frac{\sqrt{-3}}{3}\nu(x,\overline{y})$ be the corresponding Hermitian form of $\nu$.
We have the following relation:
$$
\varphi(a\otimes b, c\otimes d) = -\nu(a,c)\xi(b,d).
$$
We can choose an element $\omega\in H^1(D, \CC)_{\overline{\rho}_D}$ such that $\xi(\omega,\overline{\omega})=\sqrt{-3}$.
Precisely, we can choose a basis $E, F$ of $H^1(D,\ZZ)$, such that $E^2=F^2=0$ and $\xi(E, F)=1$, and the action of $\zeta_6$ on $H^1(D,\ZZ)$ is given by $\zeta_6(E)=E-F$ and $\zeta_6(F)=E$.
We choose $\omega$ to be $E+\zeta_3F$.
We can directly verify that $\zeta_6\cdot\omega=\zeta_6\omega=\overline{\rho}_D(\zeta_6)\omega$ and $\xi(\omega,\overline{\omega})=\sqrt{-3}$.
Hence we have 
$$
h_{\varphi}(x\otimes \omega, y\otimes \omega) = h_{\nu}(x,y).
$$
We consider the two Hermitian spaces $(T_C,h_{\nu})$ and $(T_C\otimes H^1(D, \QQ(\zeta_3))_{\overline{\rho}_D},h_{\varphi})$.
Since $\omega\in H^1(D, \QQ(\zeta_3))$, we have the following morphisms (notice that $T_C, T_F$ are defined over $\QQ(\zeta_3)$)
\begin{equation*}
	T_C\to T_C\otimes H^1(D, \QQ(\zeta_3))_{\overline{\rho}_D}\to T_F,\quad v\mapsto v\otimes \omega\mapsto \lambda^{-1}(v\otimes \omega),
\end{equation*}
with the composition respects the Hermitian forms $h_{\nu}$ on $T_C$ and $h_{\varphi}$ on $T_F$. Therefore, we have: 
\begin{prop}
\label{proposition: identification Hermitian ball}
There is a natural isomorphism between Hermitian spaces $(T_C,h_{\nu})$ and $(T_F,h_{\varphi})$. 
In particular, there is a natural isomorphism between the two complex balls $\BB(T_C)$ and $\BB(T_F)$.
\end{prop}

\section{Main Theorem}
\label{section: commu}
In this section we first study a morphism between two local systems. Then we formulate and prove our main result Theorem \ref{theorem: main}.

\subsection{Morphism between Two Local Systems}
We first state a general result that will be used.
For an algebraic group $G$ defined over a base field $k$, $G$ is called special if every principal $G$-bundle is locally trivial in Zariski topology for every reduced algebraic variety $X$ defined over $k$.
This was defined by Serre in \cite[\S 4.1]{serre1958espaces}.
Grothendieck \cite[Theorem 3]{grothendieck1958torsion} proved that an algebraic group $G$ is special if and only if $G$ is affine, connected and torsion free.
In particular, Grothendieck showed that $\GL(n)$ and $\SL(n)$ are both special.

Recall that the algebraic group $G_S$ defined in \S\ref{subseciton: GIT moduli of type I} is isomorphic to $\GL(2)$.
By Grothendieck's theorem every principal $G_S$-bundle is Zariski locally trivial.
From the Luna slice theorem (see \cite{luna1973slices}, \cite[Proposition $5.7$]{drezet2004luna}), there exists a Zariski-open subspace $U_S\subset \calF_S$ such that $\PP\calV^\circ \to \calF_S$ admits a section over $U_S$. We denote by $s\colon U_S\to \PP\calV^{\circ}$ such a section. This defines a family of K3 surfaces $\pi_{K3}\colon \calU_{K3}\to U_S$ with $\pi_{K3}^{-1}(p)=W_{s(p)}$ for each $p\in U_S$.

Recall $p_{\calF}\colon \calF_{DM}\to \calF_S$ is an isomorphism between the two GIT moduli spaces, see \S\ref{subseciton: moduli of ordered points}.
Recall from \S \ref{subsection: Review of Deligne-Mostow's Theory} we have the GIT quotient $\SF_{DM} = \SL(2)\dbs((P^9)^\circ/\FS_{\alpha})$.
By a similar argument, there is a Zariski open subspace $U_{DM}\subset \SF_{DM}$ such that $(P^9)^\circ/\FS_{\alpha}\to \SF_{DM}$ has a section $t$ over $U_{DM}$.
Without loss of generality, we assume that $U_S$ is suitably chosen such that $U_{DM}\coloneqq p_{\calF}^{-1}(U_S)$ admits a section.
We have a natural family $\pi_Q\colon \calU_Q\to U_{DM}$ (the index $Q$ stands for quotient-product surface) such that the fiber over $m\in U_{DM}$ is $\pi_Q^{-1}(m)=C_m\times_{l_d}D$. Denote by $\pi'_{K3}\colon p_{\calF}^*\calU_{K3}\to U_{DM}$ the natural families of $K3$ surfaces on $U_{DM}$.
We have the following diagram 
\begin{equation*}
\begin{tikzcd}
\calU_{K3} \arrow{d}{\pi_{K3}} & p_{\calF}^* \calU_{K3} \arrow{d}[swap]{\pi_{K3}'} & \calU_Q \arrow{dl}{\pi_Q} \\
U_S & U_{DM} \arrow{l}{p_{\calF}} & 
\end{tikzcd}
\end{equation*}
where, from Proposition \ref{proposition: birational}, the spaces $p_{\calF}^*\calU_{K3}$ and $\calU_Q$ are birational.

We can construct a global resolution $\widetilde{\calU}_Q$ of $\calU_Q$ fiberwisely, which is also a locally trivial family over $U_{DM}$. The construction is as follows. 
Recall that $\widetilde{C \times_{l_d} D}$ is the minimal resolution of $C \times_{l_d} D$, and there exists a natural birational morphism $\widetilde{C\times_{l_d}D} \to W_F$.
Let $\widetilde{\pi}_Q\colon \widetilde{\calU}_Q\to U_{DM}$ be the family of $\widetilde{C \times_{l_d} D}$. 
It admits a fiberwise birational morphism to $p_{\calF}^*\calU_{K3}$. 
Then we have two induced morphisms of local systems, namely $R^2\pi'_{K3*}\QQ(\zeta_3)\to R^2\widetilde{\pi}_{Q*}\QQ(\zeta_3)$ and $R^2\pi_{Q*}\QQ(\zeta_3)\to R^2\widetilde{\pi}_{Q*}\QQ(\zeta_3)$.

Recall that $\calT_{DM}$ is a local system on $(P^9)^\circ/\FS_{\alpha}$ which is defined in \S\ref{subsection: Review of Deligne-Mostow's Theory}.
Note that $\calT_{DM}$ restricts to a local system over $U_{DM}$ via the section $t$ with fiber a unitary Hermitian form defining the Deligne-Mostow ball (see the paragraph before Theorem \ref{theorem: deligne-mostow}).  
Over $U_{DM}$, the local system $\calT_{DM}$ can be naturally identified with a sub-local system of $R^2\pi_{Q*}\QQ(\zeta_3)$ via the following inclusion of fibers:
\begin{equation*}
(\calT_{DM})_m=T_{C_m} \cong T_{C_m}\otimes H^1(D, \QQ(\zeta_3))_{\overline{\rho}_D} \hookrightarrow H^2(C_m\times D, \QQ(\zeta_3))=(R^2\pi_{Q*}\QQ(\zeta_3))_m
\end{equation*}
We still denote by $\calT_{DM}$ its restriction to $U_{DM}\subset \calF_{DM}$. The $\mu_6$-characteristic subspaces $T_F\subset H^2(W_F,\QQ(\zeta_3))$ form a sub-local system $\calT_S\subset R^2\pi_{K3*} \QQ(\zeta_3)$ over $U_S$. We have: 
\begin{prop}
\label{proposition: morphism between local systems}
There exists an isomorphism $\theta\colon \calT_{DM}\cong p_{\calF}^*\calT_S$, such that for $[m]\in U_{DM}$, $\theta_{m}\colon T_{C_m}\to T_{F_m}$ is an isomorphism of Hermitian spaces.
\end{prop}
\begin{proof}
By Proposition \ref{proposition: characterization of T_F}, the images of $\calT_{S} \subset R^2\pi_{K3*} \QQ(\zeta_3)$ and $\calT_{DM}\subset R^2\pi_{Q*}\QQ(\zeta_3)$ in $R^2\widetilde{\pi}_{Q*}\QQ(\zeta_3)$ are the same. Then by Proposition \ref{proposition: identification Hermitian ball}, the morphism $\theta_m$ preserves the Hermitian forms.
\end{proof}

Recall that we have fixed the base point $o\in (P^N)^\circ/\FS_{\alpha}$ in \S \ref{subsection: Review of Deligne-Mostow's Theory}. 
From now on we assume that $[o]\in U_{DM}$ which gives a sextic $[F_o]\in U_S\subset \calF_S$.
Thus we have the canonical isomorphism between $\BB_{DM}$ (see \S \ref{subsection: Review of Deligne-Mostow's Theory}) and $\BB_S$ (see \S \ref{subsection: period map Prd_W}), denoted by $\theta_{o}\colon \BB_{DM}\cong \BB_S$.
The following is a direct corollary of Proposition \ref{proposition: morphism between local systems}.
\begin{cor}\label{corollary: comm}
	For $m\in (P^9)^{\circ}$, take a path $\gamma$ from $[o]$ to $[m]$ in $U_{DM}$, and it maps to a path $\gamma'$ in $U_S$, which is from $[F_m]$ to $[F_o]$.
	The path $\gamma$ (resp. $\gamma'$) defines an isomorphism $\gamma^*\colon T_{C_m}\to T_{C_o}$ (resp. $\gamma'^*\colon T_{F_m}\to T_{F_o}$).
	This induces a commutative diagram:
	\begin{equation}\label{diagram: comm}
	\begin{tikzcd}
	T_{C_m} \arrow{r}{\gamma^*} \arrow{d}{\theta_m} & T_{C_o}\arrow{d}{\theta_o} \\
	T_{F_m} \arrow{r}{\gamma'^*} & T_{F_o}.
	\end{tikzcd}
	\end{equation}
\end{cor}

\begin{prop}
\label{proposition: relation between two arithmetic subgroups}
We have
$$ \theta_o \Gamma_{DM} \theta_o^{-1}=\Gamma_S. $$
Therefore, we have an isomorphism $p_{\BB}\colon \Gamma_{DM}\bs\BB_{DM}\cong \Gamma_S\bs\BB_S$. 
\end{prop}
\begin{proof}
Take $m\in (P^9)^{\circ}$ and $b_1\in \BB_{DM}$ such that $\Prd_{DM}([m])=[b_1]\in \Gamma_{DM}\bs \BB_{DM}$.  Let $b_2=\theta_o(b_1)$. 
For every point $b_1'\in \Gamma_{DM} b_1$, there exists a path $\gamma$ in $U_{DM}$ from $[o]$ to $[m]$, such that $\gamma^*(T^{1,0}_{C_m})=b_1'$. 
Let $\gamma'$ be the image of $\gamma$ in $U_S$.
From Diagram \eqref{diagram: comm}, we have $\gamma'^*(T^{2,0}_{F_m})=\theta_o(b_1')$, which must lie in the orbit $\Gamma_S b_2$. 
Therefore, we have $\theta_o(\Gamma_{DM} b_1)\subset \Gamma_S b_2$, which is equivalent to 
\begin{equation}
\label{equation: monodromy}
(\theta_o\Gamma_{DM}\theta_o^{-1})b_2\subset \Gamma_S b_2
\end{equation}

If there exists $\tau_1\in \theta_o\Gamma_{DM}\theta_o^{-1}\backslash \Gamma_S$. By \eqref{equation: monodromy}, we have $\tau_1 b_2=\tau_2 b_2$ for some $\tau_2\in \Gamma_S$. Then $\tau_1\ne \tau_2$ and $\tau_1^{-1}\tau_2$ ($\ne \id$) is in the isotropic group of $b_2$.
If we choose $b_1$ and $b_2$ generically, then the isotropic group of $b_2$ is trivial.
Then we obtain a contradiction. We conclude that $\theta_o \Gamma_{DM} \theta_o^{-1} \subset \Gamma_S$. 
By a similar argument we have $\theta_o^{-1} \Gamma_S \theta_o \subset \Gamma_{DM}$, hence $\theta_o^{-1} \Gamma_S \theta_o = \Gamma_{DM}$.
\end{proof}

\subsection{Main Result}
\label{theorem: main result}

\begin{thm}
\label{theorem: main}
	We have the following commutative diagram:
	\begin{equation}\label{main diagram}
	\begin{tikzcd}
	\calF_{DM} \arrow{d}{p_{\calF}} \arrow[hook]{r}{\Prd_{DM}}
	& \Gamma_{DM}\bs\BB_{DM} \arrow{d}{p_{\BB}} \\
	\calF_S \arrow[hook]{r}{\Prd_S}
	& \Gamma_S\bs\BB_S \\
	\end{tikzcd}
	\end{equation}
Here $\Prd_{DM}$ is defined in \S \ref{subsection: Review of Deligne-Mostow's Theory}, $p_{\calF}$ is defined in \S \ref{subseciton: moduli of ordered points} and $\Prd_S$ is defined in \S \ref{subsection: period map Prd_W}.
\end{thm}
\begin{proof}

It suffices to show the commutativity for the open dense subspace $U_{DM}\subset \calF_{DM}$. 
Take a point $m\in (P^9)^{\circ}$ such that $[m]\in U_{DM}$. 
	We next show $(p_{\BB}\circ\Prd_{DM})([m])=(\Prd_S\circ p_{\calF})([m])$.
	We have $p_{\calF}([m])=[F_m]$. 
	Take a path $\gamma$ in $U_{DM}$ from $[o]$ to $[m]$. This path induces an isomorphism $\gamma^*\colon T_{C_m}\cong T_{C_o}$. 
	This isomorphism identifies $T^{1,0}_{C_m}$ with a positive line in $T_{C_o}$, which is the point $\Prd_{DM}([m])$ in $\BB_{DM}$.
	
	The path $\gamma$ induces a path $\gamma'$ in $U_S$ from $[F_m]$ to $[F_o]$. 
	This induces an isomorphism $H^2(W_{F_m})\cong H^2(W_{F_o})$, hence also $T_{F_m}\cong T_{F_o}$. 
	The complex line $T_{F_m}^{2,0}$ is sent to a positive line in $T_{F_o}$, which is the point $\Prd_S([F_m])$.
	We only need to check $p_{\BB}(\Prd_{DM}([m]))=\Prd_S([F_m])$.
	This is implied by Corollary \ref{corollary: comm}.
\end{proof}

\section{An Explicit Description of $(Q_F, \mu_3)$}
\label{section: abstract description}
A sextic polynomial $F\in \calV^\circ$ has an associated lattice $Q_F\cong A_2\oplus E_6(-1)^2$ with an action of $\mu_3$ (see \S \ref{subsection: period map Prd_W}). 
In this section we give an explicit description of $(Q_F,\mu_3)$.

Let $\calE=\calO_{\QQ(\zeta_3)}$ be the Eisenstein ring of integers in $\QQ(\zeta_3)$. An Eisenstein lattice $\Lambda$ is a free $\calE$-module of finite rank together with a non-degenerate Hermitian form $h\colon \Lambda\times\Lambda\to \CC$. It is called unitary if its signature is $(1, *)$.
For an Eisenstein lattice $(\Lambda, h)$, we denote by $L(\Lambda)$ the associated real lattice with the underlying abelian group $\Lambda$ and the bilinear form given by $\frac{2}{3}\Ree(h)\colon \Lambda\times\Lambda\to \RR$. We have an action of $\mu_3$ on $L(\Lambda)$ such that the action of $\zeta_3\in\mu_3$ is by multiplication of $\zeta_3$ on $\Lambda$. This group action has no nonzero fixed vectors.
For the standard rank-one lattice $\calE$ (with $h\colon \calE\times \calE\to \calE$ sending $(x,y)$ to $x\overline{y}$), the associated real lattice $L(\calE)$ is isomorphic to $A_2(\frac{1}{3})$. Here for a real lattice $R$ and a real number $a\ne 0$, we denote by $R(a)$ the lattice with same underlying space with the value of the intersection form multiplied by $a$. 

By \cite[Chapter $4$, \S $8.3$, ($120$)]{conway1999sphere}, the lattice $E_6$ is associated with an Eisenstein lattice $\Lambda_1$ of rank $3$, which is generated by the row vectors of the following generator matrix
\begin{equation*}
	M=\left(                 
	\begin{array}{ccc}   
	\sqrt{-3} & 0 & 0 \\  
	0 & \sqrt{-3} & 0 \\  
	1 & 1 & 1 \\
	\end{array}
	\right)
\end{equation*}
in $\CC^3$, and the Hermitian form is the restriction of the standard one on $\CC^3$.
The intersection matrix for $\Lambda_1$ is
\begin{equation*}
	M\overline{M}^{tr}=\left(                 
	\begin{array}{ccc}   
	3 & 0 & \sqrt{-3} \\  
	0 & 3 & \sqrt{-3} \\  
	-\sqrt{-3} & -\sqrt{-3} & 3 \\
	\end{array}
	\right).
\end{equation*}

Let $T=L(\calE(-3)\oplus \Lambda_1^2)$ be the associated real lattice with the Eisenstein lattice
\begin{equation*}
\calE(-3)\oplus \Lambda_1^2= \left( -3 \right)
    \oplus
	\left(                 
	\begin{array}{ccc}   
	3 & 0 & \sqrt{-3} \\  
	0 & 3 & \sqrt{-3} \\  
	-\sqrt{-3} & -\sqrt{-3} & 3 \\
	\end{array}
	\right)
	\oplus
	\left(                 
	\begin{array}{ccc}   
	3 & 0 & \sqrt{-3} \\  
	0 & 3 & \sqrt{-3} \\  
	-\sqrt{-3} & -\sqrt{-3} & 3 \\
	\end{array}
	\right),
\end{equation*}
which is actually integral and isomorphic to $A_2\oplus E_6(-1)^2$.
Let $L_{K3} = U^3\oplus E_8(-1)^2$ be the $K3$ lattice. 
The primitive embeddings of $A_2\oplus E_6(-1)^2$ into $L_{K3}$ are unique up to automorphisms of $L_{K3}$.
We fix one such embedding.
The standard action of $\mu_3$ on $T$ does not have non-zero fixed vectors and acts trivially on $A_T=T^*/T$.
Hence it can be extended to an action of $\mu_3$ on $L_{K3}$ with trivial restriction to $P=T^{\perp}_{L_{K3}}$.

Let $\chi\colon \mu_3\hookrightarrow \CC^{\times}$ be the tautological character.
We have a decomposition $T_{\QQ(\zeta_3)}=T_{\chi}\oplus T_{\overline{\chi}}$, where $T_{\chi}$ and $T_{\overline{\chi}}$ represent for the characteristic subspaces associated with $\chi$ and $\overline{\chi}$ respectively.
Denote by $\varphi$ the bilinear form on $T$, which naturally extends to a $\QQ(\zeta_3)$-bilinear form on $T_{\QQ(\zeta_3)}$.
We define
\begin{equation*}
h(x,y)=\varphi(x, \overline{y})
\end{equation*}
for any $x,y\in T_{\chi}$. Then $h$ is a $\QQ(\zeta_3)$-Hermitian form on $T_{\chi}$ of signature $(1,6)$.
Let $\BB(T_{\chi})$ be the complex hyperbolic ball associated with $(T_{\chi}, h)$. 

In this section we show that a generic point in $\BB(T_{\chi})$ recovers a $K3$ surface $W_F$ in a natural way, and conclude Proposition \ref{proposition: abstract model of (T_F, beta_F)}. 
\begin{prop}
\label{proposition: abstract model of (T_F, beta_F)}
For the $K3$ surface $W_F$, there exists an isomorphism $\iota\colon Q_F\stackrel{\cong}{\longrightarrow}T$ compatible with the actions of $\mu_3$ on both sides. 
In particular, $\iota(H^{2,0}(W_F))\in \BB(T_{\chi})$. 
\end{prop}
\begin{proof}
For $z\in \BB(T_{\chi})\subset \DD(L_{K3})$, by the global Torelli theorem there exists a $K3$ surface $X$ with a marking $\iota_X\colon H^2(X,\ZZ)\to L_{K3}$ and $\iota_X(H^{2,0}(X))=z$.
Recall that we have defined a $\mu_3$-action on $L_{K3}$ with the property that $L_{K3}^{\mu_3}=P\cong U\oplus A_2(-1)^3$. 
This gives rise to an automorphism $\iota_X^{-1}\circ \zeta_3\circ\iota_X$ on $H^2(X,\ZZ)$.
Note that $\iota_X^{-1}(L_{K3}^{\mu_3})=H^2(X,\ZZ)^{\mu_3}\subset \Pic(X)$. 
We take $z$ to be generic, then $\iota_X^{-1}(L_{K3}^{\mu_3})=\Pic(X)$.
We can find an element $v\in \Pic(X)_{\QQ}$ with $(v,v)>0$, such that there is no $(-2)$-vector in $\Pic(X)$ that is perpendicular to $v$. 
This implies that $v$ or $-v$ is ample. 
Since $\iota_X^{-1}(L_{K3}^{\mu_3})=\Pic(X)$, the $\mu_3$-action fixes $v$, hence also fixes an ample class. 
By global Torelli theorem, the $\mu_3$-action on $H^2(X,\ZZ)$ is uniquely induced by a $\mu_3$-action on $X$.
Next we prove $X\cong W_F$ for certain $F\in \calV^{\circ}$ in three steps.

\textbf{Step 1:} Find a nef class on $X$ that defines an elliptic fibration.

The argument here is inspired by \cite[Lemma $2.1$]{kondo1992automorphism}.
We have $L_{K3}^{\mu_3}=P \cong U\oplus A_2(-1)^3$. 
Take $e, f\in P$ with $e^2=f^2=0$ and $e\cdot f=1$.
By \cite[Chapter $8$, Remark 2.13]{huybrechts2016lectures}, we have an automorphism of $P$ that maps $(e,f)$ to $(e',f')$ such that $f'$ is nef.
We may ask $f$ to be nef.
By \cite[Chapter $2$, Proposition $3.10$]{huybrechts2016lectures}, $f$ is base-point-free and a generic member $E\in |f|$ is a smooth curve of genus $1$.
Hence the linear system of $f$ gives rise to an elliptic fibration $\pi\colon X\to \PP^1$. 
We have $(e-f)^2=-2$ and $(e-f,f)=1$. 
By Riemann-Roch theorem, we have $h^0(e-f)-h^1(e-f)+h^0(f-e)=1$, hence $h^0(e-f)+h^0(f-e)\ge 1$. Since $(f-e, f)=-1$, we have $h^0(f-e)=0$ and $h^0(e-f)\ge 1$. 
Thus $e-f$ has an effective representative $D=\sum_{i=1}^{n}k_i D_i$, where $D_i$ is an irreducible component of $D$ and $k_i>0$ for each $i$.
Since $(e-f, f)=1$ and any irreducible curve on $X$ has non-negative intersection with $f$, we may assume $(D_1, f)=1$ and $(D_i, f)=0$ for all $i\ne 1$.
Thus $D_1$ is a section of $\pi$. Without loss of generality, we assume $e=[D_1]+f$.

\textbf{Step 2:} Prove that $\mu_3$ acts on $X$ fiberwisely.

One input for this step is a classification of fixed locus of non-symplectic automorphism of order $3$ on complex $K3$ surfaces. Such a classification is obtained by \cite{artebani2008non-symplectic} and \cite{taki2011classification} independently. Reidegld \cite{reidegld2015k3} made a good summary of the results.

Denote by $X^{\mu_3}$ the fixed locus of $\mu_3$.
By \cite[Theorem $2.3$]{reidegld2015k3}, $X^{\mu_3}$ is the disjoint union of three isolated points, a smooth rational curve $s$ and a smooth curve $b$ of genus two. 
Since $\mu_3$ fixes $f\in \Pic(X)$, we know that $\mu_3$ sends a fiber to a fiber. 
Since $b$ cannot be contained in a fiber, it intersects with all fibers. Therefore, $\mu_3$ preserves every fiber.

Recall that $D_1$ is a section of $\pi$. 
Since $\mu_3$ fixes $[D_1]\in \Pic(X)$, $D_1$ must be preserved by the $\mu_3$-action. 
Each fiber $E$ is preserved by the $\mu_3$-action. 
Thus $\mu_3$ fixes the intersection point $D_1\cap E$. 
We then conclude that $\mu_3$ fixes every point on $D_1$. 
In particular, we have $s=D_1$.

\textbf{Step 3:} Analyze the type of singular fibers and conclude the proposition. 

Since every fiber of $\pi$ admits an action of $\mu_3$, the $j$-function of $\pi$ is constantly $0$. 
On each smooth fiber, the $\mu_3$-action has $3$ fixed points, with one lying on the section $s$, and the other two lying on the genus $2$ curve $b$. 
In our case, the singular fibers of $\pi$ must have $j$-invariant $0$, hence must be of Kodaira type \uppercase\expandafter{\romannumeral2}, \uppercase\expandafter{\romannumeral4}, \uppercase\expandafter{\romannumeral2}$^*$, \uppercase\expandafter{\romannumeral4}$^*$, \uppercase\expandafter{\romannumeral1}$_0^*$.
A singular fiber of type \uppercase\expandafter{\romannumeral2}$^*$, \uppercase\expandafter{\romannumeral4}$^*$ or I$_0^*$ would contribute a copy of $E_8(-1)$, $E_6(-1)$ or $D_4(-1)$ as sublattice of $A_2(-1)^3=\langle e, f\rangle ^{\perp}_{\Pic(X)}$, which is impossible.
Thus a singular fiber of type \uppercase\expandafter{\romannumeral2}$^*$, \uppercase\expandafter{\romannumeral4}$^*$ or I$_0^*$ does not appear.

The Euler number of $X$ is $e(X)=24$.
By \cite[Theorem $6.10$]{shioda1972elliptic}, we have 
\begin{equation*}
	e(X)=\sum_{p\in \Sigma}e(\pi^{-1}(p)),
\end{equation*}
where $\Sigma\subset\PP^1$ is the discriminant set.
A fiber of type \uppercase\expandafter{\romannumeral4} contributes a copy of $A_2(-1)$ in $\langle e,f\rangle^\perp$ and has Euler characteristic $4$. A fiber of type II has no contribution to $\langle e,f\rangle^\perp$ and has Euler characteristic $2$.
Thus there are $3$ fibers of type IV and $6$ fibers of type II.
These $9$ singular fibers correspond to the $9$ points on the base $\PP^1$.
The $3$ ($6$, respectively) branch points define a homogeneous polynomial $F_3(X_1,X_2)$ ($F_6(X_1,X_2)$, respectively) of degree $3$ ($6$, respectively).
We then obtain a homogeneous polynomial $F=X_0^3 F_3+F_6$ of degree $6$.
It determines a $K3$ surface $W_F$ such that $(X, \mu_3)\cong (W_F, \mu_3)$, see Remark \ref{remark: weierstrass fibration}.
Hence $(T, \mu_3)$ is isomorphic to $(Q_F, \mu_3)$. Such an isomorphism then holds for every $F\in \calV^{\circ}$.
\end{proof}

\begin{rmk}
We obtain geometrically the uniqueness of Eisenstein lattice with certain given discriminant form. We wonder whether there are more such results. 
\end{rmk}

\bibliography{ref}
\bibliographystyle{alpha}

\end{document}